\documentclass[sn-mathphys,Numbered]{sn-jnl}

\usepackage{graphicx}%
\usepackage{multirow}%
\usepackage{amsmath,amssymb,amsfonts}%
\usepackage{amsthm}%
\usepackage{mathrsfs}%
\usepackage[title]{appendix}%
\usepackage{xcolor}%
\usepackage{textcomp}%
\usepackage{manyfoot}%
\usepackage{booktabs}%
\usepackage{algorithm}%
\usepackage{algorithmicx}%
\usepackage{algpseudocode}%
\usepackage{listings}%

\usepackage{tikz}
\usepackage{bbm} 

\theoremstyle{thmstyleone}%
\newtheorem{theorem}{Theorem}
\theoremstyle{thmstyletwo}%
\newtheorem{remark}{Remark}%

\theoremstyle{thmstylethree}%
\newtheorem{corollary}{Corollary}%

\begin{document}

\title[On the  telegraph process driven by geometric counting process with Poisson-based resetting ]{On the  telegraph process driven by geometric counting process with Poisson-based resetting }


\author[1]{\fnm{Antonio} \sur{Di Crescenzo}}\email{adicrescenzo@unisa.it}
\equalcont{These authors contributed equally to this work.}

\author*[2]{\fnm{Antonella} \sur{Iuliano}}\email{antonella.iuliano@unibas.it}
\equalcont{These authors contributed equally to this work.}

\author[1]{\fnm{Verdiana} \sur{Mustaro}}\email{vmustaro@unisa.it}
\equalcont{These authors contributed equally to this work.}

\author[2]{\fnm{Gabriella} \sur{Verasani}}\email{gabriella.verasani@unibas.it}
\equalcont{These authors contributed equally to this work.}

\affil[1]{\orgdiv{Department of Mathematics}, \orgname{University of Salerno}, \orgaddress{\street{Via Giovanni Paolo II, 132}, \city{Fisciano (SA)}, \postcode{84084}, \country{Italy}}} 

\affil[2]{\orgdiv{Department of Mathematics, Computer Science, and Economics}, \orgname{University of Basilicata}, \orgaddress{\street{Viale dell'Ateneo Lucano, 10}, \city{Potenza (PZ)}, \postcode{85100}, \country{Italy}}} 


\abstract{
We investigate the effects of the resetting mechanism to the origin for a  random motion on the real line characterized by two alternating velocities $v_1$ and $v_2$. We assume that the sequences of random times concerning the motions along each velocity follow two independent geometric counting processes of intensity $\lambda$, and that the resetting times are Poissonian with rate $\xi >0$. 
Under these assumptions we obtain the probability laws of the modified telegraph process 
describing the position and the velocity of the running particle. 
Our approach is based on the Markov property of the resetting times and on the knowledge of the distribution of the intertimes between consecutive velocity changes. 
We obtain also the asymptotic distribution of the particle position when (i) $\lambda$ tends to infinity, and (ii) the time goes to infinity. In the latter case the asymptotic distribution arises properly as an effect of the resetting mechanism. A quite different behavior is observed in the two cases when $v_2<0<v_1$ and $0<v_2<v_1$. Furthermore, we focus on the determination of the 
moment-generating function and on the main moments of the process describing the particle position under reset. Finally, we analyse the mean-square distance between the process subject to resets and the same process in absence of resets. Quite surprisingly, the lowest mean-square distance can be found for $\xi =0$,  for a positive $\xi$, or for $\xi\to +\infty$ depending on the choice of the other parameters. 
}


\pacs[ORCID]{Antonio Di Crescenzo: Orcid 0000-0003-4751-7341,\\ 
Antonella Iuliano: Orcid 0000-0001-8541-8120, \\
Verdiana Mustaro: Orcid 0000-0003-4583-2612, \\
Gabriella Verasani: Orcid 0009-0000-4994-1694}


\maketitle

\section{Introduction and background}\label{sec:1}

Stochastic processes under the mechanism of resetting have attracted a growing interest in the last few years. In some instances, such processes can be represented by a diffusing particle which, at random times, is reset to a given position, which usually coincides with the starting point. Generally, the random dynamics of the process and the resetting mechanism are taken to be independent of each other, while resettings to a given position occur according to a Poisson process or a general counting process. 
\par
In this paper we deal 
with a random walker having finite speed of propagation, so that the underlying process belongs to the family  of modified telegrapher's processes instead of the ordinary diffusion processes. 
Specifically, the considered diffusive model is a suitable modification of the so-called (integrated) telegraph process. This is a continuous-time process that describes a motion on the real line,  in which the changes of directions of the two alternating velocities are governed by the Poisson process (see, for instance, Beghin et al.\ \cite{beghin2001probabilistic}). Then, many generalizations of the classical model in $\mathbb{R}$ have been proposed in the years. See for instance, Crimaldi et al.\ \cite{crimaldi2013generalized}, where any new velocity is determined by the outcome of a random trial. 
See, also, Cinque and Orsingher \cite{cinque2021exact} for the analysis of the distribution of the maximum of the asymmetric telegraph process,  
Cinque \cite{cinque2022cond} for the determination of certain related conditional distributions, and Ricciuti and Toaldo \cite{ricciuti2023semi} for a class of abstract telegraph processes generalized to the case of semi-Markov perturbations. 
\par
The wide interest on these kind of process is motivated by possible applications in physics within the area of run-and-tumble particles. Other applied fields in which such processes deserve interest include mathematical finance, geophysics, biomathematics, genetics. Detailed descriptions of motivations and applications in such areas are given in the introductions of Evans et al.\ \cite{evans2020stochastic} and Di Crescenzo et al.\ \cite{di2023some}. Moreover, in our analysis we will focus on the modified telegraph process considered in the latter paper, in which the changes of velocities are governed by the Geometric Counting Process (GCP). The choice of this underlying counting process is motivated by the need of describing the intertimes between  velocity changes through heavy tailed distributions. In fact,  the modified Pareto distribution
related to the GCP is devisable in many real phenomena where the memoryless property of the exponential intertimes of the customary Poisson process is not met. We refer, for instance, to Cha and Finkelstein \cite{cha2013note}, and Di Crescenzo and Pellerey \cite{di2019some}), for a   description of useful properties and applications of the GCP. 
\par 
Aiming to study the modified telegraph process with (i) changes of velocities governed by the GCP, and (ii)  resets  occurring through an independent Poisson process, 
hereafter we recall some useful recent contributions in the area of 
modified telegraph processes.
\par 
The literature in this field has increased rapidly in the last two or three decades. By limiting to mention the more recent investigations, we recall  the articles oriented to the telegraph process with time-dependent coefficients 
(cf.\ Angelani and Garra \cite{angelani2018probability}),  the analysis of the Wasserstein distance between Brownian motion and the telegraph process (cf.\ Barrera and Lukkarinen \cite{barrera2023quantitative}),  the telegraph process evolving  on a circle (cf.\ 
De Gregorio and Iafrate \cite{de2021telegraph}) and on an hyperbola 
(cf.\ Pogorui and  Rodr\'{\i}guez-Dagnino \cite{pogorui2023Hyperbola}). We also mention  
certain extensions oriented to the multivariate setting concerning random flights (cf.\ De Gregorio and Orsingher \cite{de2020random}), suitable fractional telegraph processes (cf.\ D'Ovidio and Polito \cite{d2018fractional}),  
the telegraph processes with jumps governed by the 
fractional alternating Poisson process (cf.\ 
Di Crescenzo and Meoli \cite{DiCrescMeoli2018}), 
the analysis of linear combinations of independent telegraph 
processes  (cf.\ Kolesnik \cite{kolesnik2018linear}), the generalized telegraph process with gamma-distributed intertimes between velocity changes 
(cf.\ Martinucci et al. \cite{martinucci2022some}).
Furthermore, some studies have been oriented to 
the telegraph process in the 2- and 3-dimensional Euclidean space with orthogonal and cyclic directions (cf.\ Orsingher et al.\ \cite{orsingher2019cyclic}), telegraph processes of the Ornstein-Uhlenbeck type (cf.\ Pogorui and Rodríguez-Dagnino \cite{pogorui2022stationary} and  
Ratanov \cite{ratanov2021ornstein}) and first-crossing-time problems for jump-telegraph processes (cf.\ Ratanov \cite{ratanov2020first}). 

\subsection{Background on reset processes}

The analysis of reset/catastrophe processes is becoming more and more appealing,   since they play a relevant role in various applied contexts, such as queueing, 
population dynamics, stochastic modeling and physics.
In this framework, birth-death processes and queueing systems subject to catastrophes have been studied in 
Di Crescenzo et al.\ \cite{di2008note}, Boudali and Economou \cite{boudali2013effect}, Dimou and Economou, \cite{dimou2013single} and Giorno et al.\ \cite{giorno2014some}, among others. 
Moreover, in Economou and Gómez-Corral, \cite{economou2007batch}, the authors study the influence of renewal generated geometric catastrophes on a population of individuals that grows stochastically according to a batch Markovian arrival process (BMAP). Another interesting application concerning the extinction of populations is investigated in Artalejo et al.\ \cite{artalejo2007evaluating} where the basic immigration process is subject to binomial and geometric catastrophes. 
Moreover, in Dharmaraja et al.\  \cite{dharmaraja2015continuous}, 
a suitable scaling limit is performed on the continuous-time Ehrenfest model with catastrophes that leads to a suitable jump-diffusion process of the Ornstein-Uhlenbeck type. 
\par
In this research area a first result related specifically to the telegraph process is given in Section 6 of Orsingher \cite{Orsingher1990}, where the author discusses a kind of annihilated process and finds certain nice connections with the Kirchoff's equations. The interpretation of the annihilation as a renewal thus leads to 
the study of telegraph processes subject to resets and restarts. Along this line,
recent results in physics are given in Masoliver \cite{masoliver2019telegraphic} and in Radice \cite{radice2021one}, where the authors study the effects of resetting mechanisms on random processes that follow the telegrapher's equation. Furthermore, an appealing discussion is present in Evans et al.\  \cite{evans2018run} where the effect of resetting on run and tumble dynamics using a renewal equation approach is investigated. 
\par
Other investigations have been oriented to the analysis of reset processes of interest in physics. 
For instance,
a finite-velocity heterogeneous diffusion process in presence of stochastic Poissonian resetting is  analysed in Sandev et al.\ \cite{sandev2022stochastic}. 
Recent studies on the Brownian motion subject to instantaneous and noninstantaneous resets 
have been presented by 
Sokolov \cite{sokolov2023linear} and Bodrova and Sokolov \cite{bodrova2020resetting}, respectively.  
Time-averaging and nonergodicity of reset geometric Brownian motion are treated in 
Vinod et al.\ \cite{vinod2022time}, \cite{vinod2022nonergodicity}. 
Similar problems are illustrated in Wang et al.\ \cite{wang2021time} for the fractional Brownian
motion and heterogeneous diffusion processes. 

\par
It should be stressed that the introduction of resets in the applications of stochastic processes is not necessarily viewed as a  phenomenon negatively affecting the system. Indeed, in various stochastic models the effect of randomly occurring resets results in a regularization of the time-course of the process. In the model under investigation, indeed, the presence of Poisson-paced resets leads to the existence of a bona-fide stationary distribution, which is not existing in absence of resets. 
\par 
\subsection{Aims of the paper}
Along the line of the mentioned studies, the aim of this paper is to investigate the properties of the telegraph process on the real line undergoing stochastic Poissonian resetting to the origin, where the numbers of displacements of the motion along each of two possible directions follow two GCP's with intensity $\lambda$. The base of our analysis is the one-dimensional stochastic process studied in the first part of Di Crescenzo et al.\ \cite{di2023some}, that describes a finite-velocity random motion in $\mathbb{R}$, with 2 velocities alternating according to a GCP. 
Moreover, in our model the resets are assumed to be instantaneous, as customary in this contexts, and as already considered in several papers mentioned above. This should be interpreted as if the resets to the origin occurs on a faster time scale with respect to the assumed alternating velocities $v_1$ and $v_2$ of the random motion, with $v_2<v_1$. 
\par
A typical approach for the study of the telegraph process and its generalization focuses on the construction and solution of partial differential equations for the probability density. 
On the contrary, our investigation will be henceforth based on the Markov property of the Poissonian resetting instants, and on the knowledge of the probability law of the process in absence of catastrophes. It is noteworthy that this allows us to obtain in closed form the probability density of the process at any finite time $t\in \mathbb{R}^+$. The study includes also the analysis of its behavior in proximity of the extremes of the diffusion interval $(v_2 t, v_1 t)$, as well as the determination of the flow function, which measures the preferential direction of the motion at any point $(x,t)$. 
Then, we focus on some asymptotic results. Indeed, we are able to obtain the limiting densities of the process when the intensity $\lambda$ tends to $+\infty$. A quite different behavior is found in the two cases when $v_2<0<v_1$ and $0<v_2<v_1$.  Indeed, in the first case the limiting density has support 
$(v_2 t, v_1 t)$ and is unimodal in 0, whereas in the second one it has support 
$(0, v_1 t)$, it is piecewise decreasing and has an upward jump in $x=v_2 t$. 
The two different behaviors are due to the fact that in the first (second) case the resets push the process toward a state, i.e.\ the origin, which belongs (does not belong) to the interval $(v_2 t, v_1 t)$, which is the support of the absolutely continuous component of the density in the absence of resets. We focus also on the stationary density of the process as $t\to +\infty$, which is found to be unimodal. Even in this case a quite different behavior is seen in the two cases $v_2<0<v_1$ and $0<v_2<v_1$. Moreover, we study also the probability density of the process subject to random initial velocity,  by mixing the cases with fixed initial velocity. 
\par
Furthermore, for a more detailed understanding of the random motion with resets we analyse the relevant moments. Such study is performed by first determining the moment-generating function (MGF) of the process conditional on the initial velocity, which is expressed by means of the generalized gamma function. This allows us to provide the moments of order one and two of the process, whose extreme values and limits are also investigated.  
\par
Finally, we focus on the mean-square distance between the processes with and without resets, that is given in terms of the second order moment in the absence of resets, which in turn is equal to the asymptotic mean-square distance for $\xi \to +\infty$. It is worth mentioning that the lowest mean-square distance strongly depends on the considered parameters, and in some cases it is attained 
even for non-trivial positive values of $\xi$. 

\subsection{Plan of the paper}
The paper is organized as follows: in Section \ref{sec:2} we recall some definitions about the GCP and we introduce the resetting mechanism for the considered one-dimensional random motion. In Section \ref{sec:4} we illustrate the main results obtained for the probability density functions (PDF) of the process, the asymptotic analysis for the PDF of the process conditional on given initial velocity,  and also conditional on random initial velocity. 
Section \ref{sec:5} is devoted to obtain the MGF of the process 
and the corresponding moments of order one and two. The mean-square distance between the processes with and without resets at zero  is also investigated. 
Finally, in Section \ref{sec:6} we draw our conclusions. 
\par
Throughout the paper,  
${\mathbbm{1}}_A$ denotes the indicator function, i.e.\ ${\mathbbm{1}}_A=1$ if $A$ is true, and ${\mathbbm{1}}_A=0$ otherwise.

\section{The distribution of intertimes and resets}\label{sec:2}

We begin this section by giving some preliminary results on the GCP. For instance,  referring to \cite{cha2013note} or Section 3.1 of \cite{di2023some}, we recall that if $\{N^{(\alpha)}(t), t \in\mathbb{R}_0^{+}\}$ is a Poisson process with intensity $\alpha$, and if $U_{\lambda}$ is an exponential distribution with mean $\lambda \in \mathbb{R}^{+}$, then 
the mixed Poisson process $\{\Tilde{N}_{\lambda}(t), t \in\mathbb{R}_0^{+}\}$ with marginal distribution  
\begin{equation}
	{\rm P}[\Tilde{N}_{\lambda}(t)=k]= \int_{0}^{t} {\rm P}\big[N^{(\alpha)}(t)=k\big]\,{\rm d}U_{\lambda}(\alpha),\quad  t \in\mathbb{R}_0^{+}, \quad k \in \mathbb{N}_{0},
	\label{eq:13}
\end{equation}
is called a GCP with intensity $\lambda$. 
We recall that its probability distribution satisfies the following properties:
\begin{itemize}
	\item[(i)]  $\Tilde{N}_{\lambda}(0)=0$;
	\item[(ii)] for all $s,t \in\mathbb{R}_0^{+}$ and $k \in \mathbb{N}_{0}$,
	\begin{equation}
		{\rm P}\{\Tilde{N}_{\lambda}(t+s)-\Tilde{N}_{\lambda}(t)=k\}= \dfrac{1}{1+\lambda s} \bigg( \dfrac{\lambda s}{1+\lambda s}\bigg)^{k}.
		\label{eq:14}
	\end{equation}
\end{itemize}
In addition, one has ${\rm E}\left[\frac{\Tilde{N}_{\lambda}(t)}{t}\right]=\lambda$ as well as for the classical Poisson process $N^{(\lambda)}(t)$, 
but in this case $\frac{\Tilde{N}_{\lambda}(t)}{t}\xrightarrow{d} Y$ as $t\to +\infty$, with $Y$ exponentially distributed with mean $\lambda$, whereas for the Poisson process one has 
$\frac{N^{(\lambda)}(t)}{t}\xrightarrow{p} \lambda$ as $t\to +\infty$. 
However, even for the GCP $\Tilde{N}_{\lambda}(t)$ the parameter $\lambda$ gives the mean frequency of occurrence of events. 
\par
Moreover, denoting by $\Tilde{T}_{n,\lambda}$ the $n$-th epoch of the process $\Tilde{N}_{\lambda}(t)$, its PDF is expressed as follows:
\begin{equation}
	f_{\Tilde{T}_{n,\lambda}}(t)=n\bigg( \dfrac{\lambda t}{1+\lambda t}\bigg)^{n-1}\dfrac{\lambda}{\big(1+\lambda t\big)^{2}}, \qquad t \in\mathbb{R}_0^{+}, \quad n\in \mathbb{N}.
	\label{eq:15}
\end{equation}
It is worth mentioning that the intertimes between two consecutive epochs, denoted by  $\Tilde{D}_{n,\lambda}=\Tilde{T}_{n,\lambda}-\Tilde{T}_{n-1,\lambda}$, with $n \in \mathbb{N}$, and 
with $\Tilde{T}_{0,\lambda}=0$, are dependent. All random durations $\Tilde{D}_{n,\lambda}$, $n \in \mathbb{N}$, have a modified Pareto distribution of type I, with marginal PDF
\begin{equation}
	f_{\Tilde{D}_{n,\lambda}}(t)=\frac{\lambda}{\big(1+\lambda t\big)^{2}}, 
	\qquad t \in\mathbb{R}_0^{+}.
	\label{eq:16}
\end{equation}
%
%
%
%
%
%
%
\par
Aiming to describe the motion of a  particle moving on the real line, let us now consider the stochastic process 
$\{({X}(t),V(t)), t \in \mathbb{R}_0^{+}\}$, with state-space $\mathbb{R}\times\{{v}_1,{v}_2\}$, 
for ${v}_1,{v}_2\in \mathbb{R}$, with ${v}_2<{v}_1$. Specifically, ${X}(t)$ and $V(t)$ describe respectively the position and the velocity of the particle at time $t$, for initial conditions  
\begin{equation}
	{X}(0)=0, \qquad V(0)= {v}_j, \quad j=1,2, 
	\label{eq:5}
\end{equation}
with 
\begin{equation}
	{X}(t)=\int_{0}^{t} V(s) \,\textnormal{d}s, \qquad V(t)=\frac{v_1+v_2}{2}
	+(-1)^{j-1}\,\frac{v_1-v_2}{2}(-1)^{N(t)}, \qquad t\in \mathbb{R}^+,
	\label{eq:6}
\end{equation}
where $j=1$ if $V(0)=v_1$, and $j=2$ if $V(0)=v_2$. 
The process $N(t)$ appearing in the second term of (\ref{eq:6}) is an alternating counting process, 
whose two subprocesses are independent GCP's. The two mentioned 
subprocesses describe the number of random intervals during which the particle runs along velocity $v_j$, for $j=1,2$. 
Clearly, the process $({X}(t),V(t))$ constitutes the extended telegraph process with underlying GCP's that has been studied extensively in the first part of the paper \cite{di2023some}. 
Specifically, in the previous paper the two independent GCP's are characterized by possibly different parameters $\lambda_1$ and $\lambda_2$ referring to the intervals during which the particle runs with velocity $v_1$ and $v_2$, respectively. 
However, in analogy to the classical telegraph process which is characterized by a single intensity and aiming to obtain more manageable expressions, from now on we assume that the parameters of the underlying GCP's are identical, i.e.\ 
$\lambda_1=\lambda_2=:\lambda$. 
\par
We recall that at every instant $t \in \mathbb{R}^+$ the position $X(t)$ is confined in $[v_2t,v_1t]$. 
Indeed, if the particle does not change velocity in $[0,t]$, then it occupies one of the extremes of $[v_2t,v_1t]$ depending on the initial velocity $V(0)=v_j$, $j=1,2$. Otherwise, if the particle changes velocity at least once in $[0,t]$, then it occupies a state in  $(v_2t,v_1t)$. 
For $t\in \mathbb{R}^{+}$, $v_2t<x<v_1t$ and $i,j=1,2$, we denote by 
\begin{equation}
	p_i(x,t \,|\, v_j)
	=\frac{1}{{\rm d}x}{\rm P}[{X}(t)\in {\rm d}x, V(t)=v_i\,|\,X(0)=0, V(0)=v_j]
	\label{eq:defpijX}
\end{equation}
the sub-densities concerning the motion described by the process $({X}(t),V(t))$. 
We remark that they can be obtained through an approach based on the analysis of the intertimes between consecutive velocity changes, so that for $t\in \mathbb{R}^{+}$ and $v_2t<x<v_1t$ one has (cf.\ Section 3.2 
of \cite{di2023some} for $\lambda_1=\lambda_2=:\lambda\in \mathbb{R}^+$)   
\begin{equation}
	p_i(x,t\,\vert \,v_j)
	=\frac{\lambda\, c_{i,j}}{(v_1-v_2)(1+\lambda t)^2}, \qquad 
	C:=(c_{i,j})=\left(
	\begin{array}{cc}
		\lambda \tau  & 1+\lambda \tau \\
		1+\lambda (t-\tau)  & \lambda (t-\tau)
	\end{array}\right),
	\label{eq:pijX}
\end{equation}
for $i,j=1,2$ and 
\begin{equation}
	\tau:=\tau(x,t)=\frac{x-v_2t}{v_1-v_2}.
	\label{eq:25}
\end{equation}
%
%
%
Other details on the properties and the probability distributions of $({X}(t),V(t))$ can be found in \cite{di2023some}. 
\par
In order to incorporate the effect of instantaneous resets, let us now study the modified process $\{(\Tilde{X}(t),V(t)), t \in \mathbb{R}_0^{+}\}$, having  state-space $\mathbb{R}\times\{{v}_1,{v}_2\}$. Here, 
the only difference with the motion described by $(X(t),V(t))$ is that now the particle undergoes instantaneous resets into state 0 according to an independent Poisson process with intensity $\xi\in \mathbb{R}^+$. Hence, the components $\Tilde{X}(t)$ and $V(t)$ describe respectively the position and the velocity of the particle at time $t$ under the new assumption.  Clearly, 
$\xi$ is the resetting rate, so that $\xi^{-1}$ is the mean time between two consecutive resetting events. We stress that the resettings to the origin 
occur instantaneously. Rather than implying that such transitions occur with infinite velocity, we assume in practice that the transition speed is much greater than the signal velocities $v_1$ and $v_2$ of the considered telegraph process. Moreover, the restarting velocity of the particle after any resetting is identical to the initial velocity $V(0)=v_j$. 
\par
As example, a sample path of $\Tilde{X}(t)$ with two resets is shown in Figure \ref{FigSamplePath}. We remark that for $\xi \to 0$ the process $\Tilde{X}(t)$ behaves as the extended telegraph process with underlying GCP's with parameter $\lambda$, i.e.\ 
the process ${X}(t)$ described so far. In brief, we can write $X(t)=\Tilde{X}(t)\vert_{\xi \to 0}$. On the other hand, $\xi \to \infty$ refers to the case in which the resets to the origin occur with infinitely high rate, so that in this case $\Tilde{X}(t)\,{\stackrel{p}{\to}} \,0$, i.e.\ $\Tilde{X}(t)$ converges to $0$ in probability. 
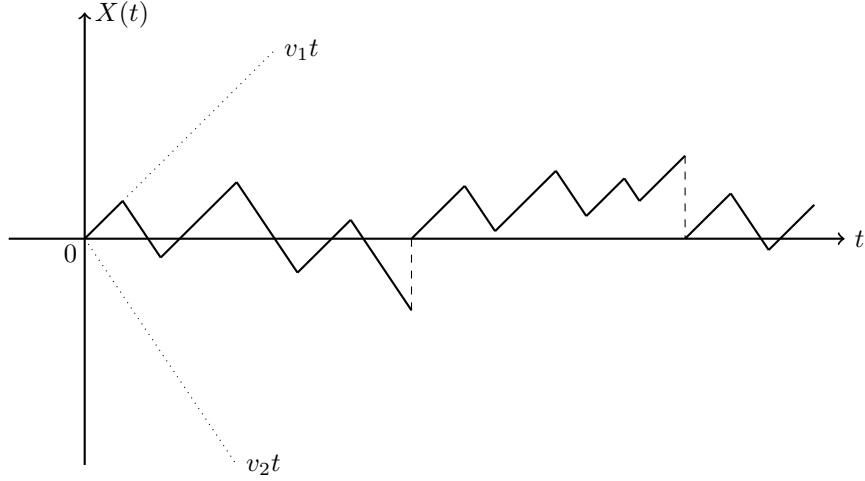
\begin{figure}[!t]
	\centering
	\begin{tikzpicture}
		\draw[->,thick] (-1,0) -- (10,0) node[anchor=west]{$t$};
		\draw[->,thick] (0,-3) -- (0,3) node[anchor=west]{$\Tilde{X}(t)$};
		\filldraw[black] (-0.4,-0.2) node[anchor=west]{$0$};
		\draw[black,thick] (0,0) -- (0.5,0.5);
		\draw[black,thick] (0.5,0.5) -- (1.0,-0.25);
		\draw[black,thick] (1.0,-0.25) -- (2.0,0.75);
		\draw[black,thick] (2.0,0.75) -- (2.80,-0.45);
		\draw[black,thick] (2.80,-0.45) -- (3.50,0.25);
		\draw[black,thick] (3.50,0.25) -- (4.30,-0.95);
		\draw[black,dashed] (4.30,-0.95) -- (4.30,0);
		\draw[black,thick] (4.30,0) -- (5.00,0.70);
		\draw[black,thick] (5.00,0.70) -- (5.40,0.10);
		\draw[black,thick] (5.40,0.10) -- (6.20,0.90);
		\draw[black,thick] (6.20,0.90) -- (6.60,0.30);
		\draw[black,thick] (6.60,0.30) -- (7.10,0.80);
		\draw[black,thick] (7.10,0.80) -- (7.30,0.50);
		\draw[black,thick] (7.30,0.50) -- (7.90,1.10);
		\draw[black,dashed] (7.90,1.10) -- (7.90,0.00);
		\draw[black,thick] (7.90,0.00) -- (8.50,0.60);
		\draw[black,thick] (8.50,0.60) -- (9.00,-0.15);
		\draw[black,thick] (9.00,-0.15) -- (9.60,0.45);
		\draw[black,dotted] (0.5,0.5) -- (2.5,2.5) node[anchor=west]{$v_1 t$};
		\draw[black,dotted] (0.,0.) -- (2.,-3.) node[anchor=west]{$v_2 t$};
		%
		%
		%
		%
	\end{tikzpicture} 
	\caption{A sample path of $\Tilde{X}(t)$, with $V(0)=v_1$.}
	\label{FigSamplePath}
\end{figure}
\par
Due to the given assumptions, under the initial condition 
$\Tilde{X}(0)=0, V(0)=v_j$ one has that the conditional probability law of 
$(\Tilde{X}(t),V(t))$,  for $t\in \mathbb{R}^{+}$ and $j=1,2$ 
is  characterized by two components:   
\begin{itemize}
	\item[(i)] a discrete component:
	\begin{equation}
		{\rm P}[\Tilde{X}(t)=v_jt, V(t)=v_j\,|\,\Tilde{X}(0)=0, V(0)=v_j],
		\label{eq:9}
	\end{equation}
	\item[(ii)] an absolutely continuous component, for $v_2t<x<v_1t$:
	\begin{equation}
		\begin{aligned}
			\Tilde{p}(x,t\,|\,v_j) & = \frac{1}{{\rm d} x}{\rm P}[\Tilde{X}(t)\in {\rm d}x \,|\,\Tilde{X}(0)=0, V(0)=v_j]\\
			& =\Tilde{p}_1(x,t\,|\,v_j)
   +\Tilde{p}_2(x,t\,|\,v_j),  
		\end{aligned}
		\label{eq:10}
	\end{equation}
	with the sub-densities $p_1$ and $p_2$  defined as 
	\begin{equation} 
		\Tilde{p}_i(x,t\,|\,v_j)\,{\rm d}x= 
		{\rm P}[\Tilde{X}(t)\in {\rm d}x, V(t)=v_i\,|\,X(0)=0, V(0)=v_j], \qquad i=1,2.
		\label{eq:11}
	\end{equation}
\end{itemize}
\par
By making use of the strong Markov property of the Poisson-paced resetting mechanism, the sub-densities (\ref{eq:11}), for $t\in \mathbb{R}^{+}$, $v_2t<x<v_1t$ and $i=1,2$ 
satisfy the following relation:
\begin{equation}
	\Tilde{p}_i(x,t \,|\, v_j)
	=e^{-\xi t}p_i(x,t \,|\, v_j)+\xi \int_{0}^t e^{-\xi s} p_i(x,s\,|\,v_j)\,{\rm d}s, 
	\qquad j=1,2, 
	\label{eq:12}
\end{equation}
where $p_i(x,t \,|\, v_j)$ are the sub-densities concerning the motion in the absence of resets defined in Eq.\ (\ref{eq:defpijX}) and provided in Eq.\	(\ref{eq:pijX}). 
Recalling that the probability that no resettings occur for time intervals greater than $t$ is $e^{-\xi t}$, the first term on the right-hand-side of Eq.\ (\ref{eq:12}) refers to the case in which the motion proceeds without resets in $[0,t]$. 
The last term of Eq.\ (\ref{eq:12}) is concerning the case when the motion experiences at least a reset in $[0,t]$, the last one occurring at time $t-s$. 
\par
Eq.\ (\ref{eq:12}) will play a considerable role in the next section for the determination of the probability law of $(\Tilde{X}(t),V(t))$. 
Moreover, it is worth mentioning that Eq.\ (\ref{eq:12}) is also 
adopted in the analysis of multidimensional diffusion processes and random walks with stochastic resetting (see, for instance, Evans et al.\ \cite{evans2014diffusion}, 
Chechkin and Sokolov \cite{chechkin2018random}, Masò-Puigdellosas et al.\ \cite{maso2019transport}). 
We recall that Eq.\ (\ref{eq:12}) has been successfully used by Majumdar et al.\ \cite{majumdar2015dynamical} and Sandev et al.\ \cite{sandev2022heterogeneous} in order to   
analyse the transition to the stationary state for long times through the large deviation function 
for the Brownian motion process and for heterogeneous diffusion processes, respectively.
Unfortunately, a similar approach is not feasible for the process under investigation, since the form of the densities (\ref{eq:pijX}) does not lead to a suitable large deviation function. 

%
\section{The probability law}\label{sec:4}

In this section, we obtain and analyse the probability distribution of the stochastic process $\{(\Tilde{X}(t),V(t)), t \geq 0\}$, i.e.\ the extended telegraph process with underlying GCP's with parameter $\lambda$ subject to a reset mechanism governed by a Poisson process with intensity $\xi$,  introduced in Section \ref{sec:2}. 
The initial condition is $\Tilde{X}(0)=0$, $V(0) = v_j$, $j=1,2$, with $v_2<v_1$. 
To avoid trivial cases, from now on we assume that $v_j\neq 0$, for $j=1,2$. 
We shall discuss the cases when $v_2 < 0 < v_1$ and $0 < v_2 < v_1$, since the other cases can be obtained by symmetry. 
\par
As already mentioned in the previous section, the probability distribution of the process is expressed in terms of an absolutely continuous component and a discrete component. As customary in the literature, the latter will be expressed in terms of the delta Dirac function $\delta(\cdot)$, so that the densities given in 
(\ref{eq:11}) and the related PDF's are viewed as generalized densities. 
\par
To ease the notation, hereafter we define the following function, 
for $x\in \mathbb{R}$ and $t\in \mathbb{R}^+$, 
\begin{equation}
	\Gamma_{\lambda}^{\xi}(x,t)=
	\begin{cases}
		\Gamma\Big[0, \Big(M_x +\frac{1}{\lambda}\Big)\xi, \Big(t+\frac{1}{\lambda}\Big)\xi\Big], \qquad \ {\rm if}\ v_2 < 0 < v_1,\\[1.5ex]
		\Gamma\Big[0, \Big(\frac{x}{v_1}+\frac{1}{\lambda}\Big) \xi, \Big(m_{x,t}+\frac{1}{\lambda}\Big)\xi\Big], \quad \ {\rm if}\ 0 < v_2 < v_1,\\
	\end{cases}
	\label{eq:19}
\end{equation}
with 
\begin{equation}
	M_x:=\max\bigg\{\frac{x}{v_1},\frac{x}{v_2}\bigg\}, \quad
	m_{x,t}:=\min\bigg\{\frac{x}{v_2},t\bigg\},
	\label{eq:20}
\end{equation}
and where 
\begin{equation}
	\Gamma(a, z_0,z_1) = \int_{z_0}^{z_1} t^{a-1} e^{-t}\,{\rm d}t
	\label{eq:GIGF}
\end{equation}
is the generalized incomplete  Gamma function. Moreover, ${\rm sgn}(v_j)$ will denote the direction of the motion concerning the velocity $v_j$, $j=1,2$.
\par
We are now able to determine the sub-densities of the process defined in 
Eq.\ (\ref{eq:11}). 
\begin{theorem}\label{theorem:4.1}
	Let $\{(\Tilde{X}(t),V(t)), \; t\in \mathbb{R}_0^{+}\}$ be the extended telegraph process with underlying GCP's with parameter $\lambda$, subject to resets to the origin with intensity $\xi$. For all $t > 0$, and $v_2t<x<v_1 t$, for $j=1,2$  we have 
	\begin{equation}
		\begin{aligned}
			\Tilde{p}_j(x,t \,|\, v_j)&=\frac{e^{-\xi t}\,\delta(x-v_jt)}{1+\lambda t}+{\mathbbm{1}}_{\{v_2 t <x <v_1t\}}\frac{e^{-\xi t}\lambda^2 \tau^{2-j}(t-\tau)^{j-1}}{(v_1-v_2)(1+\lambda t)^2}\\
			&+{\rm sgn}(v_j)\,{\mathbbm{1}}_{\{0 < \frac{x}{v_j}< t\}}\frac{\xi e^{-\xi \frac{x}{v_j}}}{v_j+\lambda x}+\boldsymbol{I}(x,t)\Bigg[(-1)^{3-j}\,\frac{\xi (v_{3-j}+\lambda x)}{(v_1-v_2)^2}\, \Theta_{\lambda}^{\xi}(x,t)\\
			&+(-1)^j \, \frac{\xi e^{\frac{\xi}{\lambda}}}{(v_1-v_2)^2}\Bigg(\frac{v_{3-j}(\lambda+\xi)}{\lambda}+x \xi\Bigg)\,\Gamma_{\lambda}^{\xi}(x,t)\Bigg],
		\end{aligned}
		\label{eq:21}
	\end{equation}
	\begin{equation}
		\begin{aligned}
			\Tilde{p}_{3-j}(x,t \,|\, v_j)&={\mathbbm{1}}_{\{v_2 t <x <v_1t\}}\frac{e^{-\xi t}\lambda[1+\lambda \tau^{j-1}(t-\tau)^{2-j}]}{(v_1-v_2)(1+\lambda t)^2}
   \\
&   +\boldsymbol{I}(x,t)\,\Bigg[(-1)^{j}\frac{\xi (v_{3-j}+\lambda x)}{(v_1-v_2)^2}\, \Theta_{\lambda}^{\xi}(x,t)
   \\
			&+(-1)^{3-j}\frac{\xi e^{\frac{\xi}{\lambda}}}{(v_1-v_2)^2}\Bigg(\frac{v_j(\lambda+\xi)}{\lambda}+ x \xi+(-1)^j\frac{\xi(v_1-v_2)}{\lambda}\Bigg)\,\Gamma_{\lambda}^{\xi}(x,t)\Bigg],
		\end{aligned}
		\label{eq:22}
	\end{equation}
	where 
	$\tau=\tau(x,t)$ and $\Gamma_{\lambda}^{\xi}(x,t)$ are defined respectively in Eqs.\ (\ref{eq:25}) and (\ref{eq:19}), with 
	\begin{equation}
		\boldsymbol{I}(x,t)={\mathbbm{1}}_{\{\min\{v_2 t,0\} < x < v_1 t\}},
		\label{eq:24}
	\end{equation}
	\begin{equation}
		\Theta_{\lambda}^{\xi}(x,t)=\left\{
		\begin{array} {ll}
			\displaystyle\frac{e^{-\xi M_x}}{1+\lambda M_x}-\frac{e^{-\xi t}}{1+\lambda t}, & {\rm if}\ v_2 < 0 < v_1,\\[2ex]
			\displaystyle\frac{e^{-\xi\frac{x}{v_1}}}{1+\lambda \frac{x}{v_1}}-\frac{e^{-\xi m_{x,t}}}{1+\lambda m_{x,t}}, & {\rm if}\ 0 < v_2< v_1.
		\end{array}
		\right.
		\label{eq:23}
	\end{equation}
\end{theorem}
\begin{proof}
	We assume that $v_2 < 0 < v_1$ and that the initial velocity is $V(0)=v_1$. Recalling Eq.\ (\ref{eq:12}) and using Eq.\ (\ref{eq:pijX})  we have
	\begin{equation}
		\Tilde{p}_1(x,t \,|\, v_1)=e^{-\xi t}p_1(x,t \,|\, v_1)+\xi \int_{0}^t e^{-\xi s}p_1(x,s \,|\, v_1)\,{\rm d}s.
		\label{eq:26}
	\end{equation}
	Then, Eq.\ (\ref{eq:26})  can be computed as follows
	\begin{equation}
		\begin{aligned}
			\Tilde{p}_1(x,t \,|\, v_1)&=\frac{e^{-\xi t}\, \delta(x-v_1t)}{1+\lambda t}+\mathbbm{1}_{\{v_2 t <x <v_1t\}}\frac{e^{-\xi t}\lambda^2 \tau(x,t)}{(v_1-v_2)(1+\lambda t)^2} \,{\rm d}s\\
			&+\xi\int_{0}^t e^{-\xi s}\,\frac{\delta(x-v_1s)}{1+\lambda s}\,{\rm d}s + \xi\int_{0}^t e^{-\xi s}\mathbbm{1}_{\{v_2 s <x <v_1s\}}\frac{\lambda^2 \tau(x,s)}{(v_1-v_2)(1+\lambda s)^2} \,{\rm d}s
		\end{aligned}
		\label{eq:27}
	\end{equation}
	where $\tau(x,s)$ is given in Eq.\ (\ref{eq:25}). Indeed, the first  term on the right-hand side of Eq.\ (\ref{eq:27}) refers to the case of no velocity changes and no resetting until time $t$, 
	whereas the second  term is concerning the case of at least one velocity change and no resetting until time $t$. 
	The last two terms indicate that the last reset occurs at time $t-s$, so that in the time interval 
	$(t-s, t)$  there are no resets. Moreover, we distinguish between the case of no velocity changes in $(t-s, t)$, in the third term on the right-hand side of Eq.\ (\ref{eq:27}), and at least one velocity change in the same interval, in the fourth term. For the first integral in Eq.\ (\ref{eq:27}), 
	recalling that $\delta(ax)=\delta(x)/|a|$, we have:
	\begin{equation}
		\xi\int_{0}^t e^{-\xi s}\,\frac{\delta(x-v_1s)}{1+\lambda s}\,{\rm d}s
		=\mathbbm{1}_{\{0 < x < v_1 t\}}\frac{\xi\, e^{-\xi \frac{x}{v_1}}}{v_1+\lambda x}.
		\label{eq:28}
	\end{equation}
	For the second integral of Eq.\ (\ref{eq:27}), we need to analyse two cases: 
	\begin{itemize}
		\item[(i)] $\frac{x}{v_1} > \frac{x}{v_2}$ if and only if $x>0$;
		\item[(ii)] $\frac{x}{v_2} > \frac{x}{v_1}$ if and only if $x<0$.
	\end{itemize}
	Recalling Eqs.\ (\ref{eq:27}) and (\ref{eq:25}), under condition (i) we have 
	\begin{equation}
		\begin{aligned}
			& \hspace{-2cm} \xi\int_{0}^t e^{-\xi s}\mathbbm{1}_{\{v_2 s <x <v_1s\}}\frac{\lambda^2 \tau(x,s)}{(v_1-v_2)(1+\lambda s)^2} \,{\rm d}s
			\\
			&=\frac{\xi\,\lambda^2}{(v_1-v_2)^2}\int_{0}^t e^{-\xi s}\mathbbm{1}_{\{v_2 s <x <v_1s\}}
			\frac{x-v_2 s}{(1+\lambda s)^2} \,{\rm d}s 
			\\
			&=\frac{\xi\,\lambda^2}{(v_1-v_2)^2}\int_{x/v_1}^t \frac{e^{-\xi s}(x-v_2 s)}{(1+\lambda s)^2} \,{\rm d}s
			\\
			&=\frac{\xi\,\lambda^2}{(v_1-v_2)^2}\Bigg\{\frac{v_2+\lambda x}{\lambda^2}\Bigg[\frac{v_1 e^{-\xi \frac{x}{v_1}}}{v_1+\lambda x}-\frac{e^{-\xi t}}{1+\lambda t}\Bigg]
			\\
			&-\frac{e^{\frac{\xi}{\lambda}}}{\lambda^2}\Bigg[x\, \xi+\frac{v_2(\lambda+\xi)}{\lambda}\Bigg]\, \Gamma\Bigg[0,\Bigg(\frac{x}{v_1}+\frac{1}{\lambda}\Bigg)\xi,\Bigg(t+\frac{1}{\lambda}\Bigg)\xi\Bigg]\Bigg\},
			%
		\end{aligned}
		\label{eq:29}
	\end{equation}
	where the last equality is obtained after some calculations and from the generalized incomplete Gamma function  (\ref{eq:GIGF}). 
	Similarly,  under condition (ii)  we have
	\begin{equation}
		\begin{aligned}
			& \hspace{-2cm} \xi\int_{0}^t e^{-\xi s}\mathbbm{1}_{\{v_2 s <x <v_1s\}}\frac{\lambda^2 \tau(x,s)}{(v_1-v_2)(1+\lambda s)^2} \,{\rm d}s
			\\
			&=\frac{\xi\,\lambda^2}{(v_1-v_2)^2}\Bigg\{
			\frac{v_2+\lambda x}{\lambda^2}\Bigg[\frac{v_2 e^{-\xi \frac{x}{v_2}}}{v_2+\lambda x}-\frac{e^{-\xi t}}{1+\lambda t}\Bigg]\\
			&-\frac{e^{\frac{\xi}{\lambda}}}{\lambda^2}\Bigg[x\, \xi+\frac{v_2(\lambda+\xi)}{\lambda}\Bigg]\, \Gamma\Bigg[0,\Bigg(\frac{x}{v_2}+\frac{1}{\lambda}\Bigg)\xi,\Bigg(t+\frac{1}{\lambda}\Bigg)\xi\Bigg]\Bigg\}.
		\end{aligned}
		\label{eq:31}
	\end{equation}
	Substituting  Eq.\ (\ref{eq:28}) in Eq.\ (\ref{eq:27}) and taking into account Eqs.\ (\ref{eq:29}) and (\ref{eq:31}), 
	we obtain the probability density $\Tilde{p}_1(x,t \,|\, v_1)$ as expressed in 
	Eq.\ (\ref{eq:21}) for $j=1$.
	The case $j=2$ of Eq.\ (\ref{eq:21}), 
	and the two cases expressed in Eq.\ (\ref{eq:22}) can be treated similarly. The only difference relies on the fact that in Eq.\ (\ref{eq:22}) there is no term concerning a discrete component, since to reach velocity $v_{3-j}$ at time $t$ when the initial velocity is $v_j$, at least a velocity change in $(0,t)$ is required. 
	Finally, when $ 0 <v_2 < v_1$ the analysis proceeds in a similar way,  
	by taking into account in Eq.\ (\ref{eq:27}) that $v_2>0$. 
\end{proof}
The explicit forms of the sub-densities obtained in Theorem \ref{theorem:4.1} 
allow us to study the behavior of $\Tilde{p}_i(x,t\,|\,v_j)$ for $x$ in proximity of 
the extremes of the diffusion interval $(v_2 t,v_1 t)$, for $t \in \mathbb{R}^+$. 
\begin{corollary}\label{corollary:4.1}
	Under the assumptions of Theorem \ref{theorem:4.1}, for $t \in \mathbb{R}^+$ and $j=1,2$, we have
	\begin{equation*}
		\begin{aligned}
			&\lim_{x \to v_j t} \Tilde{p}_j(x,t\,|\,v_j)=\frac{e^{-\xi t}}{1+\lambda t}\Bigg[\frac{\xi}{|v_j|}+\frac{\lambda^2 t}{(v_1-v_2)(1+\lambda t)}\Bigg],\\
			&\lim_{x \to v_j t} \Tilde{p}_j(x,t\,|\,v_{3-j})=\frac{e^{-\xi t} \lambda}{(v_1-v_2)(1+\lambda t)},\\
			&\lim_{x \to v_j t} \Tilde{p}_{3-j}(x,t\,|\,v_j)=\frac{e^{-\xi t} \lambda}{(v_1-v_2)(1+\lambda t)^2},\\
			&\lim_{x \to v_j t} \Tilde{p}_{3-j}(x,t\,|\,v_{3-j})={\mathbbm{1}}_{\{j=2,\,v_2>0\}}\frac{\xi e^{-\xi \frac{v_2 }{v_1}t}}{v_1+\lambda v_2 t}.
		\end{aligned}
	\end{equation*}
\end{corollary}
Clearly, the results obtained in Corollary \ref{corollary:4.1} are in agreement with those given in Remark 3 of  \cite{di2023some} when $\xi \to 0$.
\par
Let us now focus on the marginal distribution of $\{\Tilde{X}(t), t \in \mathbb{R}^+_0 \}$ conditional on $\{\Tilde{X}(0) = 0, V (0) = v_j\}$, $j = 1,2$. In this case, the absolutely continuous components for $t \in \mathbb{R}^+$ and $x \in (v_2t, v_1t)$ are described by
the PDF $\Tilde{p}(x,t\,|\,v_j)$, $j = 1,2$, introduced in (\ref{eq:10}). Moreover, 
we shall focus on the  flow function defined as
\begin{equation}
\Tilde{w}(x,t\,|\,v_j)= \Tilde{p}_1(x,t\,|\,v_j)-\Tilde{p}_2(x,t\,|\,v_j), 
\qquad j=1,2
\label{eq:flow}
\end{equation}
and given for $t\in \mathbb{R}_0^{+}$ and $v_2t < x < v_1t$, under initial condition 
$\{\Tilde{X}(0) = 0, V (0) = v_j\}$, $j = 1,2$. 
As pointed out in Orsingher \cite{Orsingher1990}, 
in a large ensemble of particles
moving as $\Tilde{X}(t)$ near $x$ at time $t$,  the function $\Tilde{w}(x,t\,|\,v_j)$ measures the excess of particles moving with velocity $v_1$ with respect to those moving with velocity $v_2$. 
\par
Recalling Eqs.\ (\ref{eq:10}) and (\ref{eq:12}) it is not hard to obtain the following results about the PDF (\ref{eq:10}) and the flow function (\ref{eq:flow}). 
\begin{theorem} \label{theorem:4.2}
	Under the assumptions of Theorem \ref{theorem:4.1}, the PDF and the flow function of the process $\{\Tilde{X}(t), \; t\in \mathbb{R}_0^{+}\}$ conditional on $\{\Tilde{X}(0)=0, V(0)=v_j\}$, for $j=1,2$ and $v_2t < x < v_1t$, are given respectively   by 
	\begin{equation}
		\begin{aligned}
			\Tilde{p}(x,t\,|\,v_j)&=\frac{e^{-\xi t}\, \delta(x-v_jt)}{1+\lambda t}+{\rm sgn}(v_j){\mathbbm{1}}_{\{0 < \frac{x}{v_j} < t\}} \frac{\xi \,e^{-\xi\frac{x}{v_j}}}{v_j+\lambda x}\\
			&+{\mathbbm{1}}_{\{v_2 t <x < v_1 t\}}\frac{e^{-\xi t} \lambda}{(v_1-v_2)(1+\lambda t)}+ \boldsymbol{I}(x,t)\frac{\xi e^{\frac{\xi}{\lambda}}}{(v_1-v_2)}\, \Gamma_{\lambda}^{\xi}(x,t),
		\end{aligned}
		\label{eq:37}
	\end{equation}
	\begin{equation}
		\begin{aligned}
			\Tilde{w}(x,t\,|\,v_j)&=(-1)^{j-1}\frac{\delta(x-v_j t)}{1+\lambda t}+{\mathbbm{1}}_{\{v_2 t < x < v_1 t\}}\frac{\lambda e^{-\xi t}\big[\lambda\big(2\tau-t)+(-1)^j\big]}{(v_1-v_2)(1+\lambda t)^2}\\
			&+{\rm sgn}(v_j)(-1)^{j-1}{\mathbbm{1}}_{\{0 < \frac{x}{v_j} < t\}}\frac{\xi e^{-\xi \frac{x}{v_j}}}{v_j+\lambda x}+\boldsymbol{I}(x,t)\Bigg\{\frac{2 \xi (v_{3-j}+\lambda x)}{(v_1-v_2)^2}\, \Theta_{\lambda}^{\xi}(x,t)\\
			&-\frac{\xi e^{\frac{\xi}{\lambda}}}{(v_1-v_2)^2}\Bigg[2 v_{3-j} \xi\Bigg(\frac{x}{v_{3-j}}+\frac{1}{\lambda}\Bigg)+(v_1+v_2)\Bigg]\Gamma_{\lambda}^{\xi}(x,t)\Bigg\},
		\end{aligned}
		\label{eq:38}
	\end{equation}
	where $\tau$, $\boldsymbol{I}(x,t)$, $\Gamma_{\lambda}^{\xi}(x,t)$ and $\Theta_{\lambda}^{\xi}(x,t)$ are defined respectively in (\ref{eq:25}), (\ref{eq:24}), (\ref{eq:19}) and (\ref{eq:23}).
\end{theorem}
\begin{proof}
	Recalling Eqs.\ (\ref{eq:21}) and (\ref{eq:22}), we obtain Eqs.\ 
	(\ref{eq:37}) and (\ref{eq:38})    after straightforward  calculations.
\end{proof}
Under the assumptions of Theorem \ref{theorem:4.2} the PDF of the process $\{\Tilde{X}(t), \; t\in \mathbb{R}_0^{+}\}$ conditional on $\{\Tilde{X}(0)=0, V(0)=v_j\}$, for $j=1,2$, satisfies
$\int_{\mathbb{R}} \Tilde{p}(x,t\,|\,v_j) \,{\rm d}x = 1$. \par
Figures \ref{pv1} and \ref{pv1pos} show some plots of 
$\Tilde{p}(x,t\,|\,v_1)$ for different time instants and velocities $(v_1,v_2)$. In both figures, one can  observe that the density possesses   a discontinuity. More specifically, being $V(0)=v_1$, when $v_2<0<v_1$ the discontinuity (having size $\frac{\xi}{v_1}$) occurs for $x=0$ (cf.\ Figure \ref{pv1}), since when the particle has no velocity changes after the last reset it only moves along positive values of $x$. Analogously, when $0<v_2<v_1$  the discontinuity (having size 
$\frac{\lambda e^{-\xi t}}{(v_2-v_1)(1+\lambda t)}$)
occurs for $x=v_2t$ (cf.\ Figure \ref{pv1pos}) as the particle moves in the interval $(0,v_2t)$ only in the  presence of resets.
\begin{figure}[!h]
    \centering
    \includegraphics[scale=0.63]{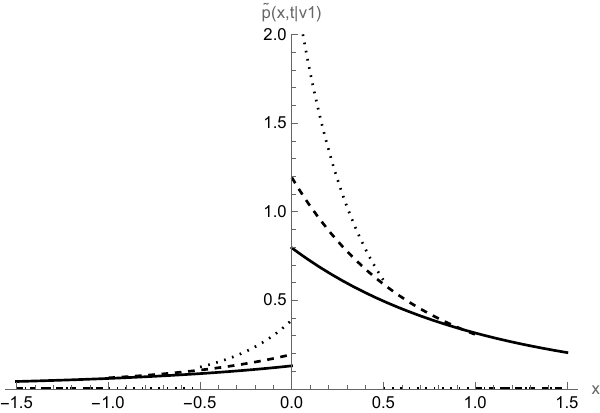}
    \includegraphics[scale=0.63]{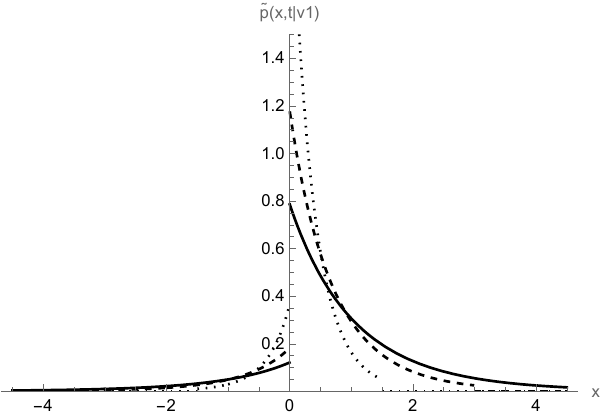}
    \caption{Plots of $\Tilde{p}(x,t\,|\,v_1)$, given in (\ref{eq:37}), for $\xi=2$, $\lambda=1$, $(v_1,v_2)=(1,-1)$ (dotted), $(2,-2)$ (dashed), $(3,-3)$ (solid) and $t=0.5$ (on the left), $1.5$ (on the right).}
  \label{pv1}
\end{figure}

\begin{figure}
    \centering
    \includegraphics[scale=0.63]{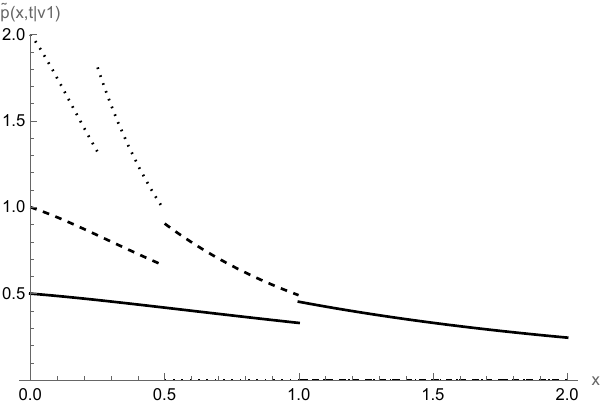}
    \includegraphics[scale=0.63]{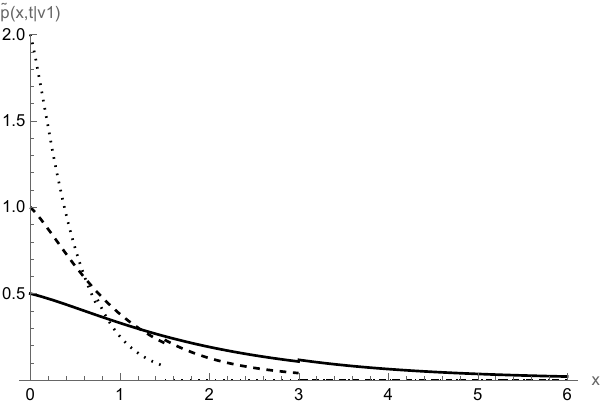}
      \caption{Plots of $\Tilde{p}(x,t\,|\,v_1)$, given in (\ref{eq:37}), for $\xi=2$, $\lambda=1$, $(v_1,v_2)=(1,0.5)$ (dotted), $(2,1)$ (dashed), $(4,2)$ (solid) and $t=0.5$ (on the left), $1.5$ (on the right).}
      
    \label{pv1pos}
\end{figure}
\subsection{Asymptotic results}\label{subsec:limits}

In this section we provide some limit results for the PDF of the process $\Tilde{X}(t)$ conditional on the initial velocity. 
\par
We remark, as recalled in Section \ref{sec:2}, that for a GCP with  intensity $\lambda$, this  parameter represents the mean frequency of occurrence of events.  Let us now discuss the behavior of density (\ref{eq:37}) when the parameter $\lambda$ goes to $+\infty$. 
For the classical telegraph process with velocity changes driven by the Poisson process it is well known that an asymptotic limit holds under the Kac's conditions, which involve both intensity and velocity, leading to the Gaussian transition function of the Brownian motion (see, e.g., Lemma 2 of Orsingher \cite{Orsingher1990}). In the forthcoming Corollary we obtain a quite different result for the process $\Tilde{X}(t)$. 

%
\begin{corollary}\label{corollary:4.2}
	Under the assumptions of Theorem \ref{theorem:4.2}, for 
	$t \in \mathbb{R}^+$, $v_2 t <x <v_1 t$ and $j=1,2$ one has
	\begin{equation}
		\pi(x,t):=
		\lim_{\lambda \to +\infty} \Tilde{p}(x,t\,|\,v_j)={\mathbbm{1}}_{\{ v_2 t < x < v_1 t\}}\frac{e^{-\xi t}}{(v_1-v_2)t}+\boldsymbol{I}(x,t)\frac{\xi }{(v_1-v_2)}\, \Gamma^{\xi}(x,t),
		\label{eq:39}
	\end{equation}
	where $\boldsymbol{I}(x,t)$ is defined in (\ref{eq:24}), and 
	\begin{equation}
		\Gamma^{\xi}(x,t)=\begin{cases}
			\Gamma\Big[0, M_x\, \xi, t\,\xi\Big], \qquad \ {\rm if}\ v_2 < 0 < v_1,\\[1.5ex]
			\Gamma\Big[0,\frac{x}{v_1}\xi, m_{x,t}\, \xi\Big], \quad \ \ {\rm if}\ 0 < v_2 < v_1,\\
		\end{cases} 
		\label{eq:40}
	\end{equation}
	is obtained from Eq.\ (\ref{eq:19}) as $\lambda \to +\infty$, with 
	$M_x$ and $m_{x,t}$ defined in (\ref{eq:20}).
\end{corollary}
It is worth mentioning that the limit density obtained in Eq.\ (\ref{eq:39}) does not depend on $j$, so that it is appropriately denoted as $\pi(x,t)$. 
This is in agreement with the remark that a large number of velocity changes relaxes the conditioning on the initial velocity. 
\par
Moreover, the two terms in the right-hand-side of (\ref{eq:39}) reflect the dependence on $\xi$. Indeed, for $\xi=0$ no reset occurs and a uniform distribution arises, in agreement with analogous results obtained in \cite{di2023some}. Instead, for $\xi>0$ the presence of resets yields a non-constant limiting density. 
\par
From Eq.\ (\ref{eq:39}) we have that if $v_2 < 0 < v_1$ then $\pi(x,t)$ is 
increasing for $x\in (v_2 t,0)$ and  decreasing for $x\in (0, v_1 t)$, so that it is unimodal 
in $x=0$ with 
$$
\pi(0,t)=\frac{1}{v_1-v_2}\left[\frac{e^{-\xi t}}{t}+\xi \,\Gamma\left(0,0,\frac{\xi}{t}\right) \right].
$$  
In the special case $v_1=-v_2>0$, from (\ref{eq:20}) one has $M_x=|x|/v_1$, so that in this case $\pi(x,t)$ is an even PDF.
Moreover, from Eq.\ (\ref{eq:39}) we have that, if $0<v_2<v_1$ then the  density $\pi(x,t)$ is 
decreasing for $x\in (0,v_2 t)$, is discontinuous (with an upward jump) in $x=v_2 t$, and 
is still decreasing for $x\in (v_2 t,v_1 t)$.  
Indeed, in this case, the probability mass in $(0,v_2 t)$ only involves the contributions due to resets, whereas the probability mass in $(v_2 t,v_1 t)$ stems from both contributions due to resets and due to diffusion. 
The above remarks are confirmed by some plots of density $\pi(x,t)$ 
shown in Figures \ref{PL} and \ref{PL_pos}
for different values of $\xi$, $t$ and $(v_1,v_2)$. 
In Figure \ref{PL} one can note also that, when the two velocities have opposite signs, the probability mass around zero increases as  $\xi$ increases. This is due to the fact that the particle is more likely to be found near the origin as the  reset rate increases. 
\par
We note that results on the first and second moment of the PDF (\ref{eq:39}) are provided  in  Remark \ref{lambdainf}, in Section \ref{sec:5}. 
\par
One of the most interesting problems in the analysis of finite-velocity random motions is 
the determination of the stationary density. Indeed, even if it does not exist for 
the classical (integrated) telegraph process, it can be obtained for suitable modified versions of such a process. We recall, for instance, Section 2 of Dhar et al.\ \cite{dhar2019} for the case of a particle
moving in an arbitrary confining potential and subjected to the telegraphic noise, and Section 4 of Crimaldi et al.\ \cite{crimaldi2013generalized} 
for the logistic steady-state density of the damped telegraph process with velocity changes governed by Bernoulli trials. 
\par
Let us now analyse the density (\ref{eq:37}) when the time $t$ goes to $+\infty$. 
Differently from the process with no resetting described in (\ref{eq:6}), which has no stationary state when $t\to +\infty$ (see  Eq.\ (67) of \cite{di2023some}), hereafter we see that the process $\Tilde{X}(t)$ admits a stationary PDF, denoted as $\Tilde{p}(x\,|\,v_j)$. 
\begin{corollary}\label{corollary:4.3}
	Under the assumptions of Theorem \ref{theorem:4.2}, 
	for $x \in \mathbb{R}$ and $j=1,2$ one has
	\begin{eqnarray}
		&& \hspace{-0.1cm} \Tilde{p}(x\,|\,v_j)
		:=\lim_{t \to +\infty} \Tilde{p}(x,t\,|\,v_j)
		\nonumber 
		\\
		&& =\left\{
		\begin{array}{ll}
			{\rm sgn}(v_j){\mathbbm{1}}_{\{\frac{x}{v_j} >0\}}
			\displaystyle\frac{\xi \, e^{-\xi \frac{x}{v_j}}}{v_j +\lambda x} 
			+\frac{\xi e^{\frac{\xi}{\lambda}}}{v_1-v_2}\Gamma\Big[0, \Big(M_x +\frac{1}{\lambda}\Big) \xi\Big], & {\rm if}\ v_2 < 0 < v_1,
			\\[3.5mm]
			{\mathbbm{1}}_{\{x>0\}}\left\{
			\displaystyle\frac{\xi \, e^{-\xi \frac{x}{v_j}}}{v_j +\lambda x} 
			+\frac{\xi e^{\frac{\xi}{\lambda}}}{v_1-v_2}\Gamma\Big[0, \Big(\frac{x}{v_1}+\frac{1}{\lambda}\Big) \xi, \Big(\frac{x}{v_2}+\frac{1}{\lambda}\Big)\xi\Big]\right\}, & {\rm if}\  0 < v_2 < v_1,
		\end{array}
		\right.
		\label{eq:41}
	\end{eqnarray}
	where $M_x$ is defined in (\ref{eq:20}).
\end{corollary}
\begin{figure}[ht!]
	\centering
	\includegraphics[scale=0.51]{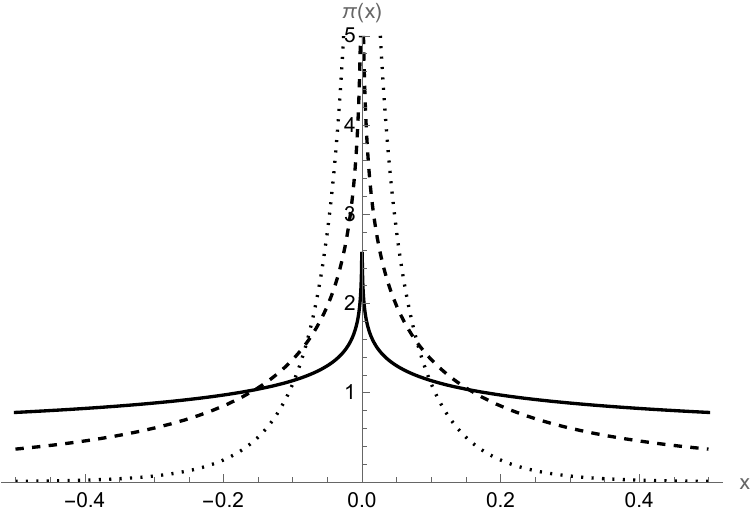}
	\includegraphics[scale=0.51]{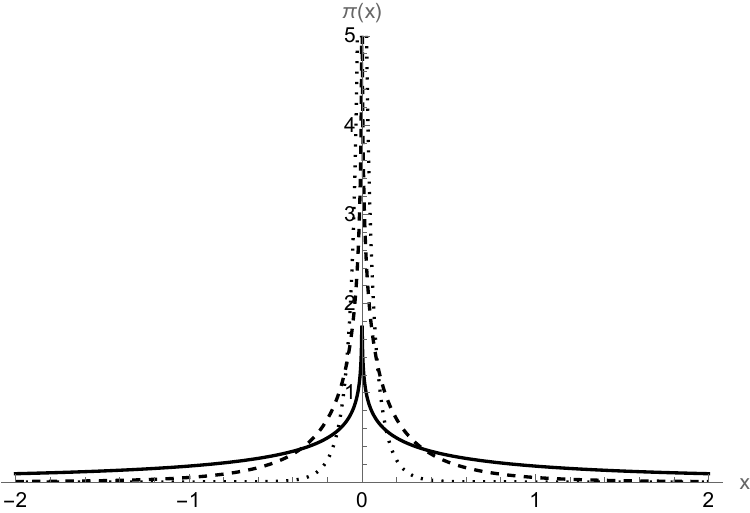}
	\caption{Plots of $\pi(x,t)$, given in (\ref{eq:39}), for $v_1=1$ and $v_2=-1$, 
		with $\xi=0.5$ (solid), $1$  (dashed), $10$ (dotted), and $t=0.5$ (left), $2$ (right).}
	\label{PL}
\end{figure}
\begin{figure}[ht!]
	\centering\includegraphics[scale=0.51]{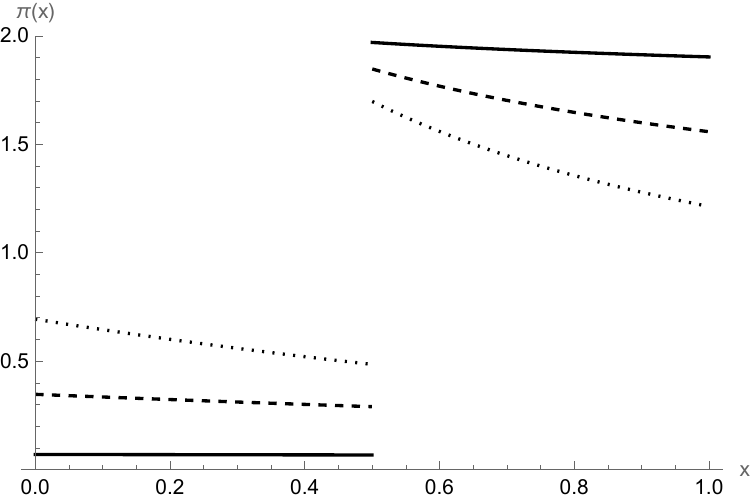}
	\includegraphics[scale=0.51]{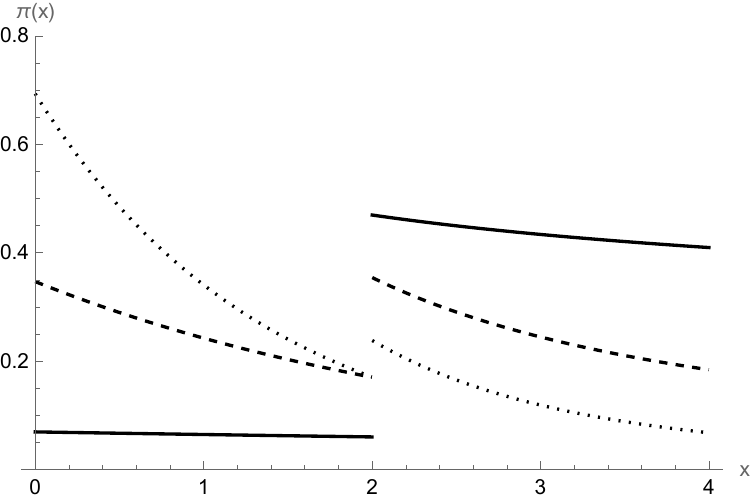}
	\caption{Plots of $\pi(x,t)$, given in (\ref{eq:39}), 
		for $v_1=2$ and $v_2=1$, with $\xi=0.5$ (solid), $1$ (dashed), $10$ (dotted), and $t=0.5$   (left), $2$ (right).}
	\label{PL_pos}
\end{figure}
\begin{figure}[ht!]
	\centering
	\includegraphics[scale=0.51]{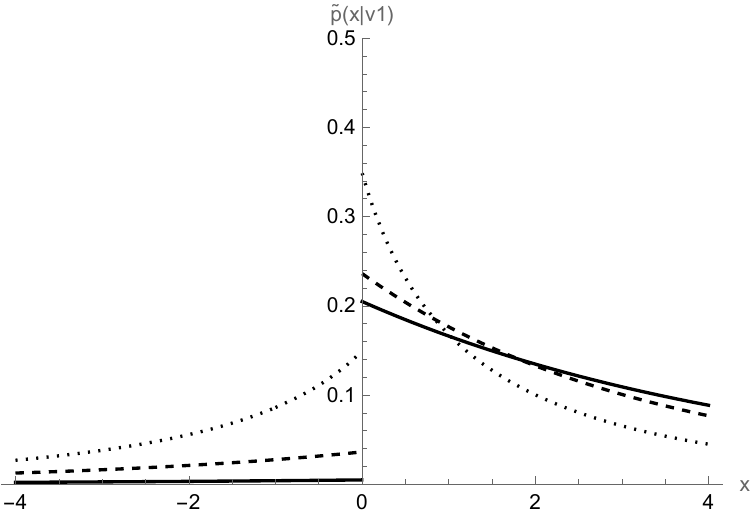}
	\includegraphics[scale=0.51]{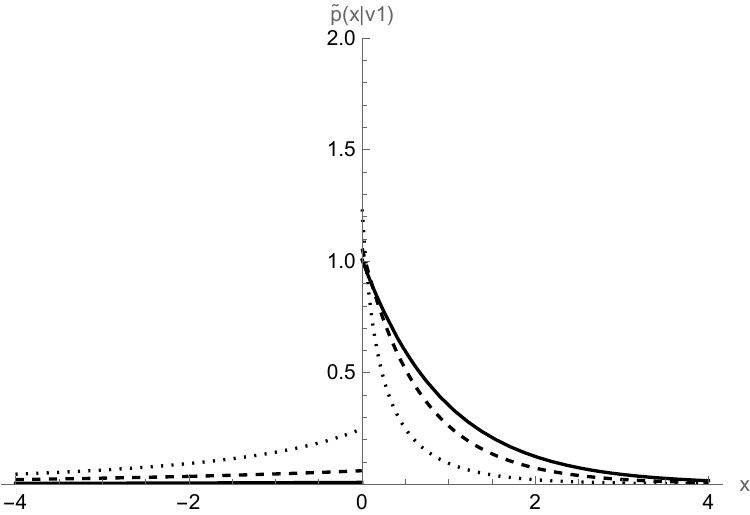}
	\includegraphics[scale=0.51]{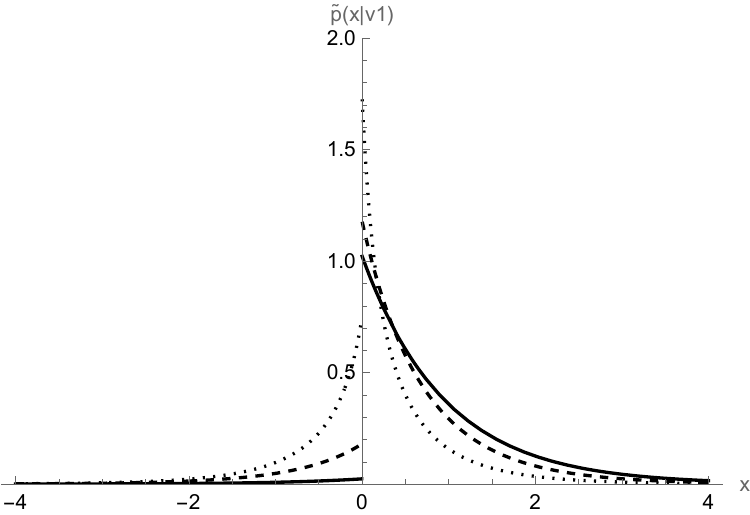}
	\includegraphics[scale=0.51]{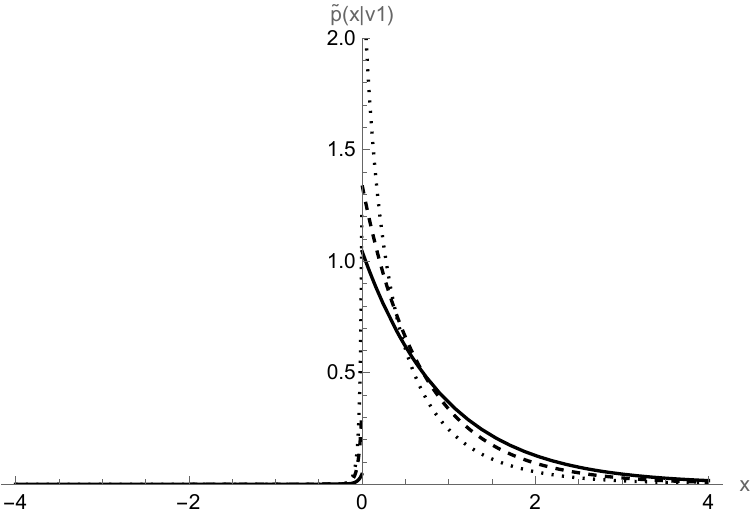}
	\caption{Plots of $\Tilde{p}(x\,|\,v_1)$, given in (\ref{eq:41}), for $\xi=2$, 
		with $\lambda=0.1$  (solid), $1$ (dashed), $10$  (dotted) and $(v_1,v_2)=(10,-10),(2,-10),(2,-2),(2,-0.1)$ (from left to right and top to bottom).}
	\label{PT_v2neg}
\end{figure}
\begin{figure}[ht!]
	\centering
	\includegraphics[scale=0.51]{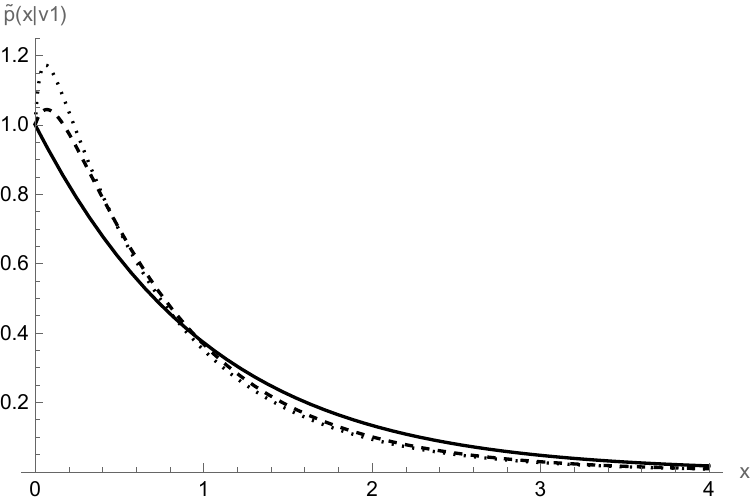}
	\includegraphics[scale=0.51]{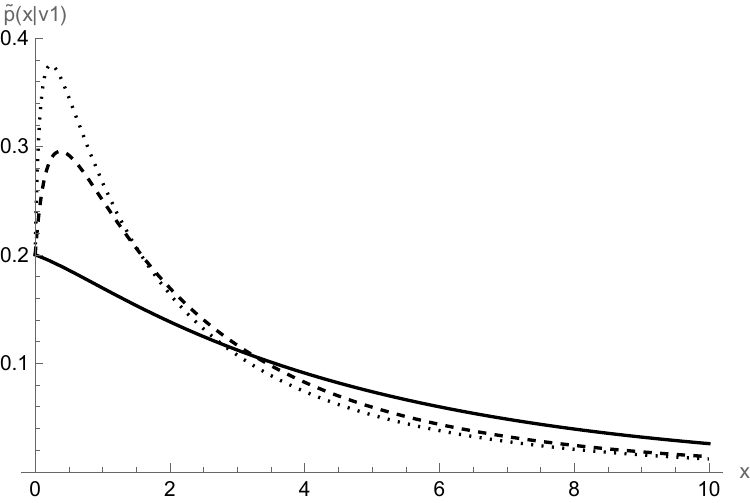}
	\caption{Plots of $\Tilde{p}(x\,|\,v_1)$, given in (\ref{eq:41}), for $\xi=2$, 
		with $\lambda=0.1$ (solid), $5$ (dashed), $20$ (dotted), $v_2=1$ and $v_1=2$ (left), $10$ (right).}
	\label{PT_v2pos}
\end{figure}
\par
Figure \ref{PT_v2neg} shows some plots of $\Tilde{p}(x\,|\,v_1)$ for different values of $\lambda$ and $(v_1,v_2)$, with $v_2<0<v_1$. It is noticeable that the probability mass in $(-\infty,0)$ decreases as $v_2$ tends to zero. We see from Figure \ref{PT_v2neg} that the limiting density $\Tilde{p}(x\,|\,v_1)$ exhibits a discontinuity in $x=0$, and that the probability mass is higher in $(0, \infty)$, since the initial velocity and the velocity after any reset is positive, i.e.\ $v_1>0$. 
The discontinuity in $x=0$ is not unusual in the field of finite-velocity random motions. Indeed, a similar behavior is found in the state-dependent telegraph process with Gamma-type intensity functions (see Section 4 of Di Crescenzo and Travaglino, 2019, \cite{di2019probabilistic}). 
On the other hand, in Figure \ref{PT_v2pos} there are some plots of $\Tilde{p}(x\,|\,v_1)$ for different values of $\lambda$ and $(v_1,v_2)$, with $0<v_2<v_1$. Here, one can observe that the density is unimodal, and has a positive maximum for higher values of $\lambda$.
\par
We finally note that results on the first and second moment of the PDF (\ref{eq:41}) are given in the next section, in point (iii) of Remark \ref{remark:5.3}. 

\subsection{Random initial velocity}\label{subsec:random}

In this section, we consider the probability distribution of $\Tilde{X}(t)$ when the initial velocity is random, according to the random variable $V_0$ such that, for $0\leq q\leq 1$,  
\begin{equation}
	{\rm P}\{V_0=v_j\}=
	\begin{cases}
		q, \qquad \ \ \, {\rm if}\  j=1,\\
		1-q, \quad {\rm if}\ \,  j=2.
	\end{cases}
	\label{eq:43}
\end{equation}
For  $t\in \mathbb{R}_0^+$, $v_2t<x<v_1t$ and $j=1,2$, we define the PDF
\begin{equation} 
	\Tilde{p}_j(x,t)\,{\rm d}x ={\rm P}[\Tilde{X}(t)\in {\rm d}x, V(t)=v_j\,|\,X(0)=0, V(0)\stackrel{d}{=}V_0],
	\label{eq:44}
\end{equation}
where `$\stackrel{d}{=}$' means equality in distribution. 
By conditioning on $V_0$, in this case we have 
\begin{equation}
	\Tilde{p}(x,t)=\frac{1}{{\rm d}x}{\rm P}[\Tilde{X}(t) \in {\rm d}x\,|\,X(0)=0, V(0)\stackrel{d}{=}V_0] 
	= q \, \Tilde{p}(x, t\, |\, v_1) + (1-q)\, \Tilde{p}(x, t\, | \, v_2).
	\label{eq:45}
\end{equation}
Making use of Eq.\ (\ref{eq:37}), we can state the following result on the distribution of $\Tilde{X}(t)$ when the initial velocity is random.
\begin{theorem}\label{proposition:4.1}
	Let the initial velocity of the process $\{(\Tilde{X}(t),V(t)), t \geq 0\}$ be distributed as 
	in Eq.\ (\ref{eq:43}). Then, for $t \in \mathbb{R}^+$, $v_2 t < x < v_1 t$, we have
	\begin{equation}
		\begin{aligned}
			\Tilde{p}(x,t)&=\Big[qH_1(x,t)+(1-q)H_2(x,t)\Big]+{\mathbbm{1}}_{\{v_2 t < x < v_1 t\}}\frac{\lambda e^{-\xi t}}{(v_1-v_2)(1+\lambda t)}\\
			&+\frac{\xi e^{\frac{\xi}{\lambda}}}{(v_1-v_2)}\boldsymbol{I}(x,t)\, \Gamma_{\lambda}^{\xi}(x,t),
		\end{aligned}
		\label{eq:46}
	\end{equation}
	with $\Gamma_{\lambda}^{\xi}(x,t)$ and $\boldsymbol{I}(x,t)$ defined respectively in (\ref{eq:19}) and (\ref{eq:24}), and 
	\begin{equation*}
		H_j(x,t)=\frac{e^{-\xi t}\,\delta(x-v_jt)}{1+\lambda t}+{\rm sgn}(v_j){\mathbbm{1}}_{\{0 < \frac{x}{v_j} < t\}}\frac{\xi e^{-\xi \frac{x}{v_j}}}{v_j+\lambda x},
		\qquad j=1,2.
	\end{equation*}
\end{theorem}
An analogous result for the flow function of $\Tilde{X}(t)$, when the initial velocity is random, 
can be obtained similarly from Eq.\ (\ref{eq:38}).
\par
In the next two corollaries, making use of Eq.\ (\ref{eq:46}), we study the asymptotic behavior of density $\Tilde{p}(x,t)$ when the intensity $\lambda$  and the time $t$ tend to $+\infty$. 
\begin{corollary}\label{corollary:4.4}
	Let the assumptions of Theorem \ref{proposition:4.1} hold. For $t \in \mathbb{R}^+$ and $v_2 t < x < v_1 t$ one has
	\begin{equation}
		\lim_{\lambda \to +\infty} \Tilde{p}(x,t)={\mathbbm{1}}_{\{v_2 t < x < v_1 t\}}\frac{e^{-\xi t}}{(v_1-v_2)t}+\frac{\xi}{(v_1-v_2)}\boldsymbol{I}(x,t)\, \Gamma^{\xi}(x,t),
		\label{eq:48}
	\end{equation}
	where $\boldsymbol{I}(x,t)$ and $\Gamma^{\xi}(x,t)$ are defined in (\ref{eq:24}) and (\ref{eq:40}), respectively.
\end{corollary}
The right-hand side of (\ref{eq:48}) shows that the distribution of the process $\Tilde{X}(t)$ when the initial velocity is random is not asymptotically uniformly distributed. Whereas, if $\xi \to 0$ the limiting PDF in (\ref{eq:48}) tends to be uniform over the diffusion interval $(v_2t,v_1t)$ as shown in Eq.\ (69) of \cite{di2023some}. 
\begin{corollary}\label{corollary:4.5}
	Let the assumptions of Theorem \ref{proposition:4.1} hold. For $t \in \mathbb{R}^+$ and $v_2 t < x < v_1 t$ one has
	\begin{eqnarray}
		&& \hspace{-0.1cm}\lim_{t \to +\infty} \Tilde{p}(x,t)
		\nonumber 
		\\
		&& =\left\{
		\begin{array}{ll}
			\Big[q\overline{H}_1(x)+(1-q)\overline{H}_2(x)\Big] 
			+\frac{\xi e^{\frac{\xi}{\lambda}}}{v_1-v_2}\Gamma\Big[0, \Big(M_x +\frac{1}{\lambda}\Big) \xi\Big], & {\rm if}\ v_2 < 0 < v_1,
			\\[3.5mm]
			\Big[q\overline{H}_1(x)+(1-q)\overline{H}_2(x)\Big]
			+\mathbbm{1}_{\{x>0\}}\frac{\xi e^{\frac{\xi}{\lambda}}}{v_1-v_2}\Gamma\Big[0, \Big(\frac{x}{v_1}+\frac{1}{\lambda}\Big) \xi, \Big(\frac{x}{v_2}+\frac{1}{\lambda}\Big)\xi\Big], & {\rm if}\  0 < v_2 < v_1,
		\end{array}
		\right.
		\nonumber 
		\\
		\label{eq:49}
	\end{eqnarray}
	where 
	\begin{equation*}
		\overline{H}_j(x)={\rm sgn}(v_j){\mathbbm{1}}_{\{\frac{x}{v_j} >0\}}\frac{\xi e^{-\xi \frac{x}{v_j}}}{v_j+\lambda x}, \qquad j=1,2.
	\end{equation*}
\end{corollary}

As before, it is easy to see that the process $\Tilde{X}(t)$, when the initial velocity is random, admits a stationary state. When $\xi \to 0$, since the effect of resets vanishes,  
the stationary density  obtained in (\ref{eq:49}) does not exist anymore. A similar result holds for the classical telegraph process driven by the Poisson process. Indeed, for this process the stationary density can be obtained under the Poisson-paced reset mechanism (see Eq.\ (27) of Masoliver \cite{masoliver2019telegraphic}) whereas it does not exist in the classical setting.
%

\section{Moment-generating functions and moments}\label{sec:5}

In this section we continue the analysis of the extended telegraph process with underlying GCP's with parameter $\lambda$, subject to resets to the origin with intensity $\xi$. 
In order to further characterize the distribution of  
$\{\Tilde{X}(t), \;t\in \mathbb{R}_0^{+}\}$ conditional on $\{\Tilde{X}(0)=0, V(0)=v_j\}$, 
for $j=1,2$, hereafter we provide the explicit form of th MGF and the moments of $\Tilde{X}(t)$. 
To this aim, we will denote by ${\rm E}_j[\cdot]$ the expectation conditional 
on $\{\Tilde{X}(0) = 0,V(0) = v_j\}$, for $j = 1,2$. Accordingly, 
\begin{equation}
	M_{j}(z,t) = {\rm E}_j\big[e^{z \Tilde{X}(t)} \big], \qquad j=1,2
	\label{eq:Mjzt}
\end{equation}
will denote the conditional MGF of $\Tilde{X}(t)$, $t\in \mathbb{R}_0^{+}$.  
\begin{theorem}\label{theorem:5.1}
	Under the assumption of Theorem \ref{theorem:4.1} for all $t \in \mathbb{R}^+$, $z\in {\cal D}$ and $j=1,2$ we have 
	\begin{equation}
		\begin{aligned}
			M_{j}(z,t) &= \frac{e^{(v_j z - \xi )t}}{1+\lambda t}+\frac{\xi}{\lambda} e^{\frac{\xi-v_j z}{\lambda}}\, G_j(z,t)\\
			&+\frac{\lambda\, e^{-\xi t}\big(e^{v_1 z t}-e^{v_2 z t})}{z(v_1-v_2)(1+\lambda t)} + \frac{\xi}{z(v_1-v_2)} \, \sum_{k=1}^2 (-1)^{k-1}\, e^{\frac{\xi-v_i z}{\lambda}} \, G_k(z,t),
		\end{aligned}
		\label{eq:50}
	\end{equation}
	where 
	$$
	{\cal D} = \left\{
	\begin{array}{ll}
		\displaystyle\left(\frac{\xi}{v_2} , \frac{\xi}{v_1} \right), & \hbox{if }\; v_2 < 0 < v_1, 
  \\[4mm]
		\displaystyle\left(-\infty,\frac{\xi}{v_2}\right), & \hbox{if }\; 0 < v_2 < v_1,
	\end{array}
	\right.
	$$
	and
	\begin{equation*}
		G_j(z,t)=\Gamma\Big[0, \frac{\xi-v_j z}{\lambda}, \frac{(\xi-v_j z)(1+\lambda t)}{\lambda}\Big], 
		\qquad j=1,2.
	\end{equation*}
\end{theorem}
\begin{proof}
	Let us consider the case when the initial velocity is $V(0)=v_1$, under the condition $v_1 < 0 < v_2$. 
	Recalling Eqs.\ (\ref{eq:37}) and (\ref{eq:Mjzt}), we have
	\begin{equation*}
		\begin{aligned}
			M_1(z,t)&
			=\int_{-\infty}^{+\infty} e^{zx} \, \Tilde{p}(x,t\,|\,v_1) \,{\rm d}x\\
			&=\frac{e^{(v_1 z-\xi) t}}{1+\lambda t} + \xi \int_{0}^{v_1 t}\frac{e^{zx-\xi \frac{x}{v_1}}}{(v_1+\lambda x)} \,{\rm d}x 
			+\frac{\lambda}{(v_1-v_2)(1+\lambda t)}\int_{v_2 t}^{v_1t} e^{z x-\xi t}\,{\rm d}x
			\\
			&+\frac{\xi e^{\frac{\xi}{\lambda}}}{(v_1-v_2)}
			\Bigg\{\int_{v_2 t}^0 e^{z x}\Gamma\bigg[0,\bigg(\frac{x}{v_2}+\frac{1}{\lambda}\bigg) \xi,\bigg(t+\frac{1}{\lambda}\bigg)\xi\bigg]\,{\rm d}x\\
			&+\int_{0}^{v_1 t} e^{z x}\, \Gamma\bigg[0,\bigg(\frac{x}{v_1}+\frac{1}{\lambda}\bigg) \xi,\bigg(t+\frac{1}{\lambda}\bigg)\xi\bigg]\,{\rm d}x\Bigg\}.\\
		\end{aligned} 
	\end{equation*}
	After some direct calculations we obtain
	\begin{equation*}
		\begin{aligned}
			M_1(z,t)&=\frac{e^{(v_1 z-\xi) t}}{1+\lambda t} + \frac{\xi\, e^{\frac{\xi-v_1 t}{\lambda}}}{\lambda}\Gamma\bigg[0,\frac{\xi-v_1 t}{\lambda},\frac{(\xi-v_1 t)(1+\lambda t)}{\lambda}\bigg]\\
			&+\frac{\lambda\, e^{-\xi t}}{(v_1-v_2)(1+\lambda t)}\bigg(\frac{e^{v_1 t z}-e^{v_2 t z}}{x}\bigg)\\
			&+\frac{\xi \, e^{\frac{\xi-v_1 t}{\lambda}}}{(v_1-v_2)z}\, \Gamma\bigg[0,\frac{\xi-v_1 t}{\lambda},\frac{(\xi-v_1 t)(1+\lambda t)}{\lambda}\bigg]\\
			&-\frac{\xi \, e^{\frac{\xi-v_2 t}{\lambda}}}{(v_1-v_2)z}\, \Gamma\bigg[0,\frac{\xi-v_2 t}{\lambda},\frac{(\xi-v_2 t)(1+\lambda t)}{\lambda}\bigg] 
		\end{aligned}
	\end{equation*}
	such that $\xi-v_1 z >0$ and $\xi-v_2 z >0$, which imply 
	$z\in\left(\frac{\xi}{v_2} , \frac{\xi}{v_1} \right)$. 
	Along a similar line we can determine  the MGF $M_1(z,t)$  when 
	$0 <v_1 <v_2$, and the  MGF $M_2(z,t)$ in both cases $v_2 < 0 < v_1$ and $0 < v_1 <v_2$. 
	Eq.\ (\ref{eq:50}) thus follows.
\end{proof}
It is not hard to verify from the MGF (\ref{eq:50}) that $M_j(0,t)=1$ for all 
$t\in\mathbb{R}_0^{+}$. 
Now, making use of Eq.\ (\ref{eq:50}) we can obtain the first and second conditional moments of the process $\{\Tilde{X}(t), \; t\in \mathbb{R}_0^{+}\}$. 
\begin{theorem}\label{theorem:5.2}
	Under the assumptions of Theorem \ref{theorem:5.1}, 
	for all $t \in \mathbb{R}^+$ and $j=1,2$, we have 
	\begin{equation}
		\begin{aligned}
			{\rm E}_j[\Tilde{X}(t)]&=\big(1-e^{-\xi t}\big)\Big(\frac{v_1+v_2}{2 \xi}+\frac{v_j-v_{3-j}}{2 \lambda}\Big)\\
			&-\frac{\xi \, e^{\frac{\xi}{\lambda}}\big(v_j-v_{3-j}\big)}{2 \lambda^2}G_\xi(t)+\frac{t\, e^{-\xi t}(v_j-v_{3-j})}{2(1+\lambda t)}
		\end{aligned}
		\label{eq:54}
	\end{equation}
	\begin{equation}
		\begin{aligned}
			{\rm E}_j[\Tilde{X}^2(t)]&=\big(1-e^{-\xi t}\big)\Big[\frac{2\big(v_1^2+v_1v_2+v_2^2\big)}{3\xi^2}+\Big(\frac{2v_j^2-v_1 v_2- v_{3-j}^2}{3 \lambda}\Big)\Big(\frac{1}{\xi}-\frac{1}{\lambda}\Big)\Big]
   \\
			&-t e^{-\xi t}\,\left[ \frac{2(v_1^2+v_1v_2+v_2^2)}{3\xi}+\frac{2v_j^2-v_1 v_2-v_{3-j}^2}{3\lambda(1+\lambda t)}\right]
   \\
   & +\frac{\xi e^{\frac{\xi}{\lambda}}(2v_j^2-v_1 v_2-v_{3-j}^2)}{3\lambda^3}\,G_\xi(t)\\
		\end{aligned}
		\label{eq:55}
	\end{equation}
	where
	\begin{equation}
		G_\xi(t)=\Gamma\Big[0, \frac{\xi}{\lambda},\frac{\xi}{\lambda}\big(1+\lambda t\big)\Big].
		\label{eq:56}
	\end{equation}
\end{theorem}
\begin{proof}
	By making use of Eq.\ (\ref{eq:50}) and recalling that, for $j = 1,2$,
	\begin{equation*}
		{\rm E}_j[\Tilde{X}^n(t)]=\lim_{z \to 0} \frac{{\rm d}^n}{{\rm d}z^n}M_j(z,t),
		\qquad n=1,2 
	\end{equation*}
	the announced results thus follow.  
\end{proof}
\begin{remark} \label{maxmin}
Making use of Theorem \ref{theorem:5.2} we can determine the extreme values of the conditional mean ${\rm E}_j[\Tilde{X}(t)]$, $j=1,2$. Indeed, 
due to Eq.\ (\ref{eq:54}), one has 
$$
\frac{{\rm d}}{{\rm d} t}{\rm E}_j[\Tilde{X}(t)]
=\frac{e^{-\xi t}[2v_j+\lambda t (2+\lambda t)(v_1+v_2)]}{2(1+\lambda t)^2}, \qquad j=1,2,
$$
so that \\
	$(i)$ if $0<v_1<-v_2$, then ${\rm E}_1[\Tilde{X}(t)]$ has a maximum for 
	$$
     t=t_M:=-\frac{1}{\lambda}\left(1+\frac{\sqrt{v_2^2-v_1^2}}{v_1+v_2}\right),
 $$
	$(ii)$ if $0<-v_2<v_1$, then ${\rm E}_2[\Tilde{X}(t)]$ has a minimum for 
	$$
      t=t_m:=-\frac{1}{\lambda}\left(1-\frac{\sqrt{v_1^2-v_2^2}}{v_1+v_2}\right). 
    $$
\end{remark}

Figure \ref{v2} shows some instances of ${\rm E}_j[\Tilde{X}(t)]$, for different values of $v_1$ and $v_2$. Note that maximum and minimum values are visible in the plots when $v_1$ and $v_2$ are chosen according to the conditions given in Remark \ref{maxmin}.
\begin{figure}[!t]
	\centering
	\includegraphics[scale=0.46]{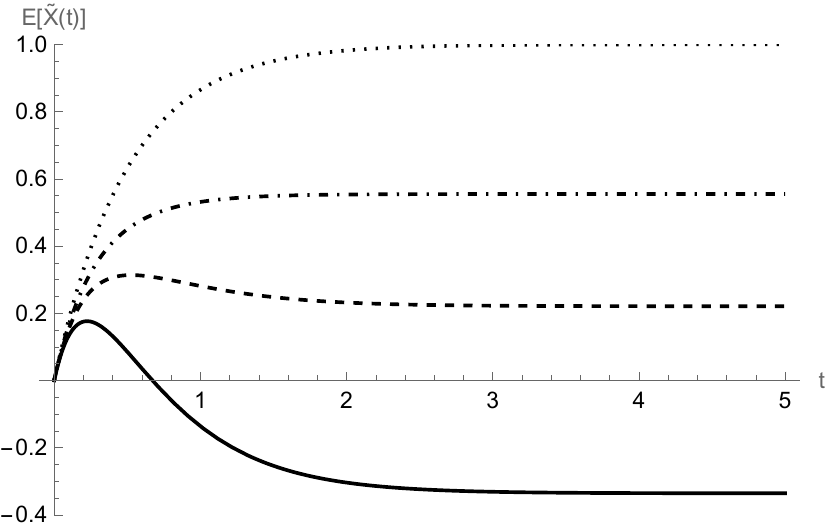}
	\includegraphics[scale=0.46]{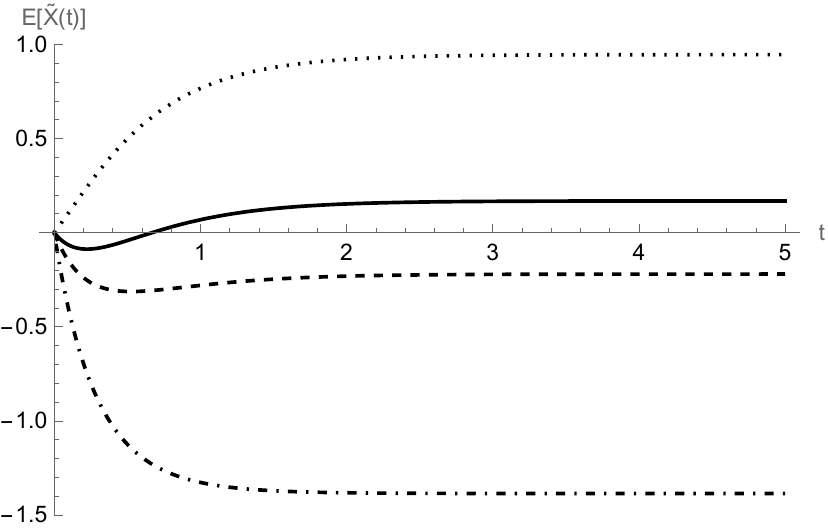}
	\caption{Plots of $ {\rm E}_j[\Tilde{X}(t)]$ for $j=1$ (on the left), $j=2$ (on the right), $\xi=2, \;\lambda=1,\; v_1=2$ and different values of $v_2=2,-2,-5,-10$ (from top to bottom) in the left side; for $\xi=2, \;\lambda=1,\; v_1=5$ and different values of $v_2=1,-1,-2,-5$ (from top to bottom) in the right side.}
	\label{v2}
\end{figure}
\begin{remark}\label{remark:5.4}
	Under the assumptions of Theorem \ref{theorem:5.2}, for $t \in \mathbb{R}^+_0$ one has the following symmetry relation:
	\begin{equation}
		{\rm E}_1[\Tilde{X}(t)]+{\rm E}_2[\Tilde{X}(t)]
  =(v_1+v_2)\Big(\frac{1-e^{-\xi t}}{\xi}\Big).
		\label{eq:59}
	\end{equation}
 The right-hand-side of Eq.\ (\ref{eq:59}) tends to $(v_1+v_2)t$ as $\xi \to 0$, 
 in agreement with Eq.\ (47) of \cite{di2023some}.
\end{remark}
\par
Let us now discuss some limit behaviors of the moments ${\rm E}_j[\Tilde{X}^n(t)]$, for $j=1,2$ and $n=1,2$. 
\begin{remark}\label{lambdainf}
	Under the assumptions of Theorem \ref{theorem:5.2}, for $t\in \mathbb{R}^+$ and $j=1,2$ one has:
\begin{equation}
\lim_{\lambda \to +\infty}{\rm E}_j[\Tilde{X}(t)]
   =\frac{v_1+v_2}{2}\Big(\frac{1-e^{-\xi t}}{\xi}\Big),
\label{lim_lambda_mom}
\end{equation}
 and 
 $$
	\lim_{\lambda \to +\infty}{\rm E}_j[\Tilde{X}^2(t)]
   =\frac{2(v_1^2+v_1v_2+v_2^2)}{3\xi^2}[1-e^{-\xi t}(1+\xi t)].
$$
 Clearly, when $\lambda$ goes to $+\infty$ the dependence on the initial velocity vanishes, so that the right-hand-side of Eq.\ (\ref{lim_lambda_mom}) is equal to one half of the right-hand-side of Eq.\ (\ref{eq:59}). 
\end{remark}

\begin{remark}\label{remark:5.3}
(i)	Under the assumptions of Theorem \ref{theorem:5.2}, for $t\in \mathbb{R}^+$ and $j=1,2$  one has:
	\begin{equation}
		\begin{aligned}
			\lim_{\xi \to 0}{\rm E}_j[\Tilde{X}(t)]
   &=\frac{(v_1+v_2)t}{2}+\frac{(v_j-v_{3-j})t}{2(1+\lambda t)},
   \\
			\lim_{\xi \to 0}{\rm E}_j[\Tilde{X}^2(t)]
   &=\frac{\big(v_1^2+v_1v_2+v_2^2\big)t^2}{3}+\frac{\big(2v_j^2-v_1v_2- v_{3-j}^2\big)t^2}{3 (1+\lambda t)}.
			\label{lim_xi_mom2}
		\end{aligned}
	\end{equation}
	These results correspond to those obtained in Theorem 2 of \cite{di2023some} for initial velocity  $V(0)=v_1$ and $\lambda_1=\lambda_2$. 
 \\
 (ii) Furthermore, for   $j=1,2$ and $n=1,2$ we have 
	\begin{equation*}
		\begin{aligned}
			\lim_{\xi \to +\infty}{\rm E}_j[\Tilde{X}^n(t)]=0.
		\end{aligned}
	\end{equation*}
 The latter result is in accordance with intuition. Indeed, if the intensity of the resets increases significantly, the process is frequently brought back to the origin. 
 \\
 (iii) Finally, under the same assumptions we get
	\begin{equation*}
		\begin{aligned}
			\lim_{t \to +\infty}{\rm E}_j[\Tilde{X}(t)]&=\frac{v_1+v_2}{2\xi}+\frac{v_j-v_{3-j}}{2\lambda}-\frac{\xi e^{\frac{\xi}{\lambda}}(v_j-v_{3-j})}{2\lambda^2}\Gamma\left(0,\frac{\xi}{\lambda}\right),
   \end{aligned}
   \end{equation*}
 \begin{equation}
		\begin{aligned}
			\lim_{t \to +\infty}{\rm E}_j[\Tilde{X}^2(t)]&=\frac{2(v_1^2+v_1v_2+v_2^2)}{3\xi^2}+\frac{2v_j^2-v_1v_2-v_{3-j}^2}{3\lambda}\left[\frac{1}{\xi}- \frac{1}{\lambda} \right] \\
			&+\frac{\xi e^{\frac{\xi}{\lambda}}(2v_j^2-v_1v_2-v_{3-j}^2)}{3\lambda^3}\Gamma\left(0,\frac{\xi}{\lambda}\right). 
    \label{eq:58}
    \end{aligned}
	\end{equation}
\end{remark}
\begin{figure}[t!]
	\centering
	\includegraphics[scale=0.7]{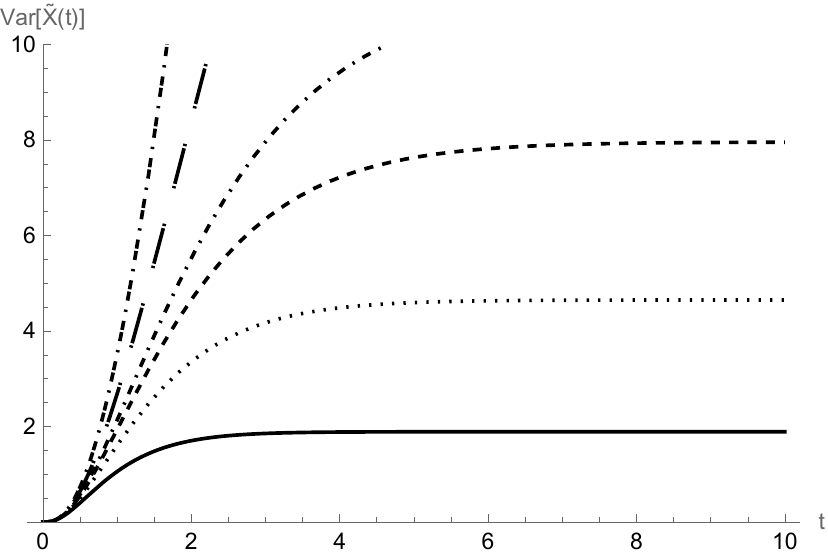}
	\caption{Plots of $ {\rm Var}[\Tilde{X}(t)]$ for $\lambda=1,\; v_1=2, \; v_2=-4$ and different values of $\xi=0.1,0.5,0.85,1,1.3,2$ (from top to bottom).}
	\label{varianze_fig}
\end{figure}
%
\par
It is worth noting that Eq.\ (\ref{eq:58}) shows that the 
mean square displacement of $\Tilde{X}(t)$, i.e.\ the moment of the second order, 
tends to a constant as $t$ goes to $+\infty$. This behaviour is quite different from 
other similar models. 
For instance, from Eq.\ (4.9) of Masò-Puigdellosas et al.\ \cite{maso2019transport} 
for a Gaussian propagator it follows that 
the asymptotic limit of the mean square displacement depends heavily on the 
decay parameter of the Pareto-distributed resetting times. 

\par
We are now able to investigate the behavior of the variance of the process. 
\begin{remark}\label{varianza}
The variance of $\Tilde{X}(t)$ conditional on $\{\Tilde{X}(0) = 0,V(0) = v_j\}$, 
$j = 1,2$, can be obtained by making use of Eqs.\ (\ref{eq:54}) and (\ref{eq:55}). 
Its expression is omitted being quite complex. 
However, we denote it by ${\rm Var}[\Tilde{X}(t)]$ since such conditional variance 
does not depend on the initial velocity $v_j$. 
In this case, whatever the initial velocity we have
	\begin{align}
		\lim_{\xi \to 0}{\rm Var}[\Tilde{X}(t)]&=\frac{\lambda t^3 (4+\lambda t)(v_1-v_2)^2}{12(1+ \lambda t)^2},
	\end{align}
	\begin{align} \label{limitevar}
		\lim_{t \to +\infty}{\rm Var}[\Tilde{X}(t)]&=\frac{5v_1^2+2v_1v_2+5 v_2^2}{12 \xi^2}+\frac{(v_1-v_2)^2}{6\lambda \xi}+\frac{-11v_1^2+10v_1v_2+v_2^2}{12 \lambda^2} \nonumber\\
		&+\frac{e^{\frac{\xi}{\lambda}}(v_1-v_2)(3\lambda(v_1+v_2))+\xi(7v_1-v_2)}{6 \lambda^3}\Gamma\left( 0, \frac{\xi}{\lambda}\right) 
        \nonumber\\
		&-\frac{e^{\frac{2\xi}{\lambda}}\xi^2(v_1-v_2)^2}{4\lambda^4}\left[\Gamma\left( 0, \frac{\xi}{\lambda}\right)  \right]^2.
	\end{align}
\end{remark}
Some instances of ${\rm Var}[\Tilde{X}(t)]$ are shown in Figure \ref{varianze_fig} 
for different values of $\xi$. 
It can be seen from Eq.\ (\ref{limitevar}) that, when $\xi$ tends to zero, the variance diverges with respect to $t$ (see also the case $\xi=0.1$ in Figure \ref{varianze_fig}). This is in accordance with Eq.\ (45) of  \cite{di2023some}, where it is shown that, for the process without resets $X(t)$, the variance is divergent. 
\subsection{Mean-square distance}\label{subsec:mean}

Here, we refer to a  probability metric as a functional that measures a suitable distance between the processes $\Tilde{X}(t)$ and $X(t)$, introduced in Section \ref{sec:2}. 
In other terms, we aim to measure   the mean-square distance between the process subject to resets and the same process without resetting. 
To this aim, for  $\xi> 0$, $t \in \mathbb{R}_0^+$ and $j=1,2$ we define
\begin{equation}
	\begin{aligned}
		\Delta_{j}^{\lambda}(\xi,t)&:= {\rm E}_j\big[\lvert\Tilde{X}(t)-X(t)\rvert^2\big]\\
		&={\rm E}_j\big[\Tilde{X}^2(t)\big]+{\rm E}_j\big[X^2(t)\big]-2\, {\rm E}_j\big[\Tilde{X}(t)\big]\, {\rm E}_j\big[X(t)\big],
	\end{aligned}
	\label{eq:61}
\end{equation}
where ${\rm E}_j$ represents the mean conditional on initial velocity $v_j$ for both processes $\Tilde{X}(t)$ and $X(t)$, $j=1,2$. 
Clearly, from Eq.\ (\ref{eq:61}) one  has 
$$
 \Delta_{j}^{\lambda}(0,t)
 := \lim_{\xi\to 0}\Delta_{j}^{\lambda}(\xi,t)
  =2\,{\rm Var}[X(t)],
$$
where such variance does not depend on  $j$. Indeed, making use of the results given in Section 3.3 of \cite{di2023some}, we have 
$$
{\rm Var}[X(t)]
=\frac{\lambda t^3(4+\lambda t)(v_1-v_2)^2}{12(1+\lambda t)},
\qquad V(0)=v_j, \; j=1,2 
$$
and
\begin{equation}
	{\rm E}_j[X^2(t)]
	=\frac{t^2[3v_j^2+\lambda t(v^2_1+v_1v_2+v_2^2)]}{3(1+\lambda t)},
	\qquad  j=1,2.
	\label{eq:63}
\end{equation}
\begin{theorem}\label{proposition:5.1}
	Under the assumptions of Theorem \ref{theorem:5.2}, the mean-square distance defined in Eq.\ (\ref{eq:61}), for $\lambda>0$ and $t \in \mathbb{R}_0^+$ is given by 
	\begin{equation}
		\Delta_{j}^{\lambda}(\xi,t)
		={\rm E}_j[X^2(t)]+
  A_{j}^{\lambda}(\xi,t),
		\qquad j =1,2
		\label{eq:60}
	\end{equation}
	where
	\begin{equation}
		\begin{aligned}
			A_{j}^{\lambda}(\xi,t)
			& =  -t e^{-t \xi}\Bigg[\frac{2v_j^2-v_1 v_2-v_{3-j}^2}{3 \lambda(1+\lambda t)}+\frac{2(v^2_1+v_1v_2+v_2^2)}{3 \xi}+\frac{t(v_j-v_{3-j})[2v_{j} + \lambda t(v_1+v_2)]}{2(1+\lambda t)^2}\Bigg]
			\nonumber
			\\
			& +  (1-e^{-\xi t})\Bigg[\frac{2(v^2_1+v_1v_2+v_2^2)}{3 \xi^2}+\frac{2v_j^2-v_1 v_2-v_{3-j}^2}{3 \lambda}\bigg(\frac{1}{\xi}-\frac{1}{\lambda}\bigg)    
			\nonumber
			\\
			&  -\frac{t(2v_{j} + \lambda t(v_1+v_2))}{2(1+\lambda t)}\bigg(\frac{v_j-v_{3-j}}{\lambda}+\frac{v_1+v_2}{\xi}\bigg)\Bigg]
			\nonumber
			\\
			& +    \frac{e^{\frac{\xi}{\lambda}}\xi}{\lambda^2}G_\xi(t)\Bigg[\frac{2v_j^2-v_1 v_2-v_{3-j}^2}{3 \lambda}+\frac{t(v_j-v_{3-j})(2v_{j} + \lambda t(v_1+v_2))}{2(1+\lambda t)}\Bigg], 
		\end{aligned}
	\end{equation}
	for $G_\xi(t)$  given in (\ref{eq:56}), and ${\rm E}_j[X^2(t)]$ 
	shown in (\ref{eq:63}). 
\end{theorem}
\begin{proof}
	Recalling Eqs.\ (\ref{eq:54}) and (\ref{eq:55}), 
	after straightforward calculations one obtains Eq.\ (\ref{eq:60}) from Eq.\ (\ref{eq:61}). 
\end{proof}
\begin{corollary}\label{remark:5.5}
	From Theorem \ref{proposition:5.1}  we have
	\begin{equation*}
		\lim_{\xi \to +\infty} \Delta_{j}^{\lambda}(\xi,t)
		={\rm E}_j[X^2(t)],
		\qquad j=1,2.
	\end{equation*}
	where   ${\rm E}_j[X^2(t)]$ is given in (\ref{eq:63}). 
\end{corollary}
\begin{proof}
	It immediately follows by noting that 
	$\lim_{\xi \to +\infty}A_{j}^{\lambda}(\xi,t)=0$. 
\end{proof}
We note that the result expressed in Corollary \ref{remark:5.5} arises from 
(\ref{eq:61}) since  $\Tilde{X}(t)$ converges to zero 
as $\xi \to +\infty$, i.e.\ it converges to zero in the presence of instantaneous resets occurring with rate growing to $+\infty$. 
\par
Some plots of $\xi \mapsto \Delta_{j}^{\lambda}(\xi,t)$ are shown in Figures \ref{dist1} and \ref{dist2} for different values of $\lambda$.
Specifically, Figure \ref{dist1} shows some trends when the initial velocity is $V(0)= {v}_1$, with $|v_1|>|v_2|$. If the velocities have discordant signs (see Figure \ref{dist1}, left case) then there is a single minimum. The same result is obtained when $V(0)={v}_2$, with $|{v}_2|>|{v}_1|$.
If $V(0)={v}_2$ with $|{v}_1|>|{v}_2|$, then the distance $\Delta_{j}^{\lambda}(\xi,t)$ is monotonic decreasing in $\xi$. This is justified by the fact that the presence of resets causes the initial velocity to prevail. This is also confirmed by Figure \ref{dist2}, left case.
If velocities ${v}_1$ and ${v}_2$ have the same sign, then the distance $\Delta_{j}^{\lambda}(\xi,t)$ is monotonic increasing in $\xi$, whatever the initial speed (see  Figures \ref{dist1} and \ref{dist2}, right case).
\begin{figure}[ht!]
	\centering
	\includegraphics[scale=0.72]{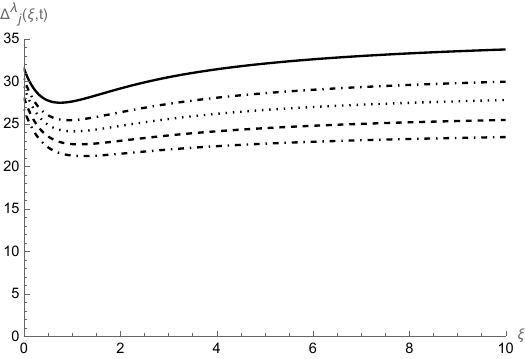}
	\includegraphics[scale=0.72]{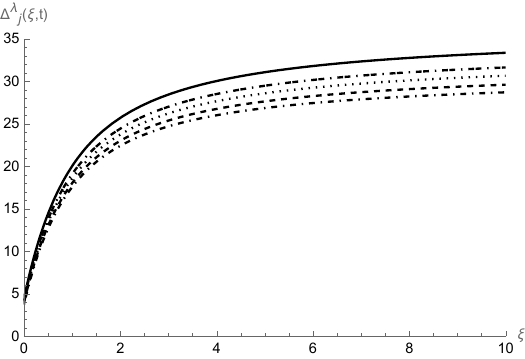}
	\caption{Plots of $\Delta_{j}^{\lambda}(\xi,t)$, with $j=1$, for $t=3$, $v_1=3$, $v_2=-1$ and $\lambda=1,1.5,2,3,5$ (from top to bottom) in the left side; for $t=3$, $v_1=2.5$, $v_2=1$ and $\lambda=1,1.5,2,3,5$ (from top to bottom) in the right side.}
	\label{dist1}
\end{figure}
\begin{figure}[ht!]
	\centering
	\includegraphics[scale=0.72]{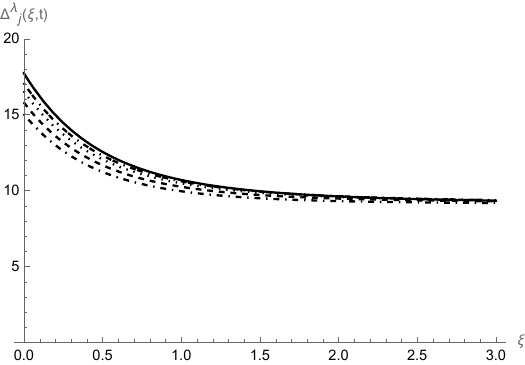}
	\includegraphics[scale=0.72]{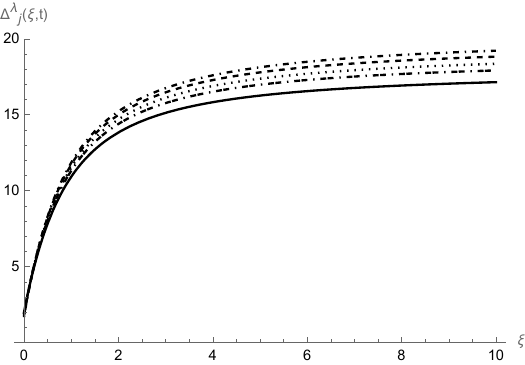}
	\caption{Plots of $\Delta_{j}^{\lambda}(\xi,t)$, with $j=2$, for $t=3$, $v_1=2$, $v_2=-1$ and $\lambda=1,1.5,2,3,5$ (from bottom to top) in the left side; for $t=3$, $v_1=2$, $v_2=1$ and $\lambda=1,1.5,2,3,5$ (from bottom to top) in the right side.}
	\label{dist2}
\end{figure}
%
\section{Concluding remarks}\label{sec:6}

This paper has been focused on the study of the telegraph process
such that (i) the particle velocity alternates cyclically between $v_1$ and $v_2$ according to a GCP, and (ii) the particle position is reset randomly in time with a constant rate $\xi$, this corresponding to Poissonian resetting to the initial position.

\par 
The presence of intertimes having Pareto-type distributions yields results that are quite different from those concerning the standard telegraph process with exponential intertimes.

\par
The main effects of the Poissonian resetting is the insurgence of a steady-state limit for the process as $t\to +\infty$. In particular, the steady-state density  has  a remarkable different behavior in the cases $v_2<0<v_1$ and $0<v_2<v_1$. Moreover, 
it is worth mentioning that the mean and the variance of the process possess a finite limit as $t$ tends to $+\infty$. This effect arises due to presence of the resets, and indeed it is more evident when the reset rate $\xi$ increases. 

\par 
The investigation has been focused, in conclusion, also to the analysis of the  mean-square distance between the process subject to resets and the same process without resetting. In some instances, for $t$ fixed, the mean-square distance at time $t$ is monotonic in $\xi$, whereas 
in certain cases it possesses a minimum or a maximum at a finite value of $\xi$.   
\par
Moreover, we observe that the process $\Tilde{X}(t)$ admits a stationary density when the initial velocity is random. However, when $\xi$ tends to $0$ the stationary PDF does not exist, as in the classical telegraph process driven by the Poisson process. In a similar way, the stationary density  exists also for the telegraph process with Poisson-paced velocity changes with instantaneous resets (cf.\ Masoliver \cite{masoliver2019telegraphic} and Evans et al.\ \cite{evans2018run}) 
and with non-instantaneous resets (see Radice \cite{radice2021one}).
\par

The novelty of the proposed model is mainly based on the fact that it provides one of the first cases in which the reset phenomenon is studied for an extended telegraph processes with non-exponential but Pareto-distributed times between consecutive velocity changes. 
Moreover, it is remarkable to stress that the approach based on the Poissonian renewals of the resets rather than the analysis of the partial differential equations governing the motion allowed us to come to the non-trivial results obtained in the foregoing, such as 
\\
-- the existence of closed-form and tractable expressions of the transient densities,  the flow function and the main moments of the process, 
\\
-- the determination of the limiting densities and moments obtained as the intensity of the process and the time go to infinity, 
\\
-- the analysis of the extremal values of the mean-square distance between the processes with and without resets. 

\par

We point out that future developments of the present investigation can be oriented to the 
analysis of the first-passage-time problem for the considered process through a suitable 
barrier (see, e.g.\ Bodrova and Sokolov \cite{bodrova2020resetting}). 
For instance, this is aimed to study the mean first-passage-time, i.e.\ 
the average time necessary to hit a specified target, since this quantity is of large interest in biomathematics, for instance
in foraging problems for randomly moving animals.
In addition, the stochastic model exploited herewith and the 
related first-passage-time problems deserve interest for  
\\
-- the analysis of run-and-tumble particle motion perturbed by Gaussian white
noise or subjected to confining potentials,
\\
-- the description of the alternating behaviour 
of the prices of risky assets typically observed in financial markets,
\\
-- the modeling of finite-velocity random motions arising in geophysical phenomena 
related to the earthquakes occurrences and the vertical motions of volcanic regions. 
\par

More detailed comments and a list of useful references about the above mentioned applied 
contexts are provided in the Section 1 of \cite{di2023some}.
\par
In conclusion, we depict a further  source of  ideas for future studies and applications. Indeed, specific interest toward  reset processes arises also in the area of robotic search problems, where single or teams of sensor-enabled mobile robots collaborate for the random exploration and the control of fields under observation. The extended telegraph process with underlying GCP's and Poissonian resets is a prototype model for the description of units moving through random alternations within the search for a mobile target. 
In this setting, in the absence of target identification the reset and restart is viewed (i) as a return to the origin caused by the near exhaustion of the resources of the robot, followed by a resource refill, or (ii) as a replacement of a damaged unit 
by a new devise that starts again the search from the origin. The use of the underlying GCP instead of the classical Poisson process, in this setting, allows to take into account random times between changes of directions of the robots characterized by heavy tails. 

\backmatter

%
%
%

\section*{Acknowledgments}

The authors are members of the group GNCS of INdAM (Istituto Nazionale
di Alta Matematica). 
This work is partially supported by PRIN 2022, project ``Anomalous Phenomena on Regular and Irregular Domains: Approximating Complexity for the Applied Sciences", PRIN 2022 PNRR, project ``Stochastic Models in Biomathematics and Applications", 
and by INdAM-GNCS,
Project ``Modelli di shock basati sul processo di conteggio geometrico e
applicazioni alla sopravvivenza''(CUP E55F22000270001).

\section*{Data Availability}
 All data generated or analysed during this study are included in this published article.

\section*{Declarations}
All authors contributed equally to this paper.
The authors declare that they have no known competing financial interests or personal relationships that could have appeared to influence the work reported in this paper.

\bibliography{sn-bibliography}

\end{document}